\documentclass{amsart}

\usepackage{amsmath,amsfonts,amsthm,amssymb,amscd}
\usepackage{hyperref}

\usepackage[alphabetic]{amsrefs}
\usepackage{tikz}
\usepackage{graphicx}
\usepackage{tikz-cd}
\usepackage{xcolor}

\usepackage{chngcntr}
\usepackage{apptools}

\usepackage[shortlabels]{enumitem}

\usetikzlibrary{decorations.markings}
\usetikzlibrary{patterns}

\newtheorem{theorem}{Theorem}[section]
\newtheorem{lemma}[theorem]{Lemma}
\newtheorem{proposition}[theorem]{Proposition}
\newtheorem{corollary}[theorem]{Corollary}
\newtheorem*{claim}{Claim}

\theoremstyle{definition}
\newtheorem{definition}[theorem]{Definition}
\theoremstyle{question}
\newtheorem{ques}[theorem]{Question}

\theoremstyle{remark}
\newtheorem{example}[theorem]{Example}
\newtheorem{remark}[theorem]{Remark}

\numberwithin{equation}{section}

\newcommand{\abs}[1]{\lvert#1\rvert}

\newcommand{\mF}{\mathcal {F}}

\newcommand{\ugd}{\underline{\mathrm{gd}}}
\newcommand{\ucd}{\underline{\mathrm{cd}}}
\newcommand{\vcd}{{\mathrm{vcd}}}
\newcommand{\cd}{{\mathrm{cd}}}

\newcommand{\ue}{{\underline E}}

\newcommand{\omminu}{\Omega_J'}
\newcommand{\B}{B}
\newcommand{\D}{K}

\newcommand{\Br}{B^R}

\newcommand{\Z}{\mathbb Z}
\newcommand{\mQ}{\mathcal Q}
\newcommand{\C}[2]{C_{#1}(#2,R)}

\newcommand{\orb}{\mathcal{O}_{\mF}G}
\newcommand{\orbmod}{\mbox{Mod--}\mathcal{O}_{\mF}G}

\newcommand{\nathom}{\mathrm{Hom}_{\mF}}

\newcommand{\rH}{\widetilde{H}}

\newcommand{\pan}[2]{{#1}_{#2}}

\newcommand{\flink}[2]{\cup_{#2 < #2'}\pan{#1}{#2'}}

\newcommand{\rlink}[2]{\cup_{#2 < #2'}\rpan{#1}{#2'}}
\newcommand{\rpan}[2]{{#1_{#2}}}

\newcommand{\E}[1]{\underline{E}{#1}}

\begin{document}

\tikzset{->-/.style={decoration={
  markings,
  mark=at position .65 with {\arrow{>}}},postaction={decorate}}}

\title[Bestvina complex for group actions]{Bestvina complex for group actions with\\ a strict fundamental domain}

\author{Nansen Petrosyan}

 \thanks{Both authors were supported by the EPSRC First Grant EP/N033787/1. T.\ P.\  was partially supported by (Polish) Narodowe Centrum Nauki, grant no.~UMO-2015/18/M/ST1/00050.}

\author{Tomasz Prytu{\l}a}
\address{School of Mathematics, University of Southampton, Southampton SO17 1BJ, UK}
\email{n.petrosyan@soton.ac.uk}
\email{t.p.prytula@soton.ac.uk}

\subjclass[2010]{Primary 20F65, 05E18, 05E45; Secondary 20E08, 20J06}

\date{\today}

\keywords{complex of groups, classifying space, standard development, Coxeter system, building, virtual cohomological dimension, Bredon cohomological dimension}

\begin{abstract} We consider a strictly developable simple complex of finite groups $G(\mathcal Q)$. We show that Bestvina's construction for Coxeter groups applies in this more general setting to produce a complex that is equivariantly homotopy equivalent to the standard development. When $G(\mathcal Q)$ is non-positively curved, this implies that the Bestvina complex is a cocompact classifying space for proper actions of $G$ of minimal dimension. As an application, we show that for groups that act properly and chamber transitively on a building of type $(W, S)$, the dimension of the associated Bestvina complex is the virtual cohomological dimension of $W$.  
We give further examples and applications in the context of Coxeter groups, graph products of finite groups, locally $6$--large complexes of groups and groups of rational cohomological dimension at most one. Our calculations indicate that, because of its minimal cell structure, the Bestvina complex is well-suited for cohomological computations.\end{abstract}

\maketitle

\setlist[enumerate,1]{before=\itshape,font=\normalfont, leftmargin=1.1cm}

\section{Introduction}\label{sec:intro}

 For a discrete group $G$, a proper $G$--CW--complex is a $G$--CW--complex with only finite cell stabilisers.  A proper $G$--CW--complex $X$ is said to be a \emph{model for a classifying space for proper actions} $\ue G$, if for any finite subgroup $F$ of $G$, the fixed point set 
$X^F$ is contractible.  Such a complex $X$ always exists and is unique up to equivariant homotopy equivalence.   The minimal dimension of any model for $\ue G$ is denoted by $\ugd G$ and is called the {\it geometric dimension for proper actions} of $G$.
The algebraic counterpart of the geometric dimension is the Bredon cohomological dimension $\ucd G$. The relation between these two invariants is analogous to the one between cohomological dimension of a group $G$ and the  minimal dimension of an Eilenberg-Mac~Lane space $K(G, 1)$.  It can be 
shown that $\ucd G= \ugd G$ except when there exist groups $G$ for which $\ucd G=2$ and $\ugd G=3$~\cite{LuckMeintrup, BLN}. 
The Bredon cohomological dimension of $G$ is an upper bound for the cohomological dimension of any torsion-free subgroup of $G$. In particular, for groups that are virtually torsion-free, the virtual cohomological dimension $\vcd G$ always satisfies $\vcd G\leqslant \ucd G$.

The main motivation to study $\ue G$ comes from the Isomorphism Conjectures (see e.g.,\ \cite{BCH}, \cite{Lucksurvey}). Other applications of $\ue G$ include computations in group cohomology and the formulation of a generalisation from finite to infinite groups of the Atiyah-Segal Completion Theorem in topological $K$--theory  (see \cite[\S7-8]{Lucksurvey}).
With these applications in mind, it is always desirable to have models for $\ue G$ with good geometric properties, such as for example, non-positively curved, cocompact, of minimal dimension and cell structure.

In \cite{Best}, for any finitely generated Coxeter system $(W, S)$, Bestvina constructed an acyclic polyhedral complex $B(W,S)$ of dimension equal to $\vcd W$, on which $W$ acts as a reflection group, properly and cocompactly. The same construction produces a contractible $B(W,S)$ with $\dim B(W,S)= \vcd W$ except possibly when $\vcd W=2$. 
In fact, we show that $B(W,S)$ is equivariantly homotopy equivalent to the Davis complex $\Sigma_W$. Therefore $B(W,S)$ is a model for $\ue W$ of minimal dimension. In the main part of the paper, we derive an analogous result in the more general setting of strictly developable simple complexes of finite groups.


A  \emph{simple complex of finite groups} $G(\mathcal Q)$ over a poset $\mathcal Q$ consists of a family of finite groups $\{P_J\}_{J \in \mathcal Q}$ 
 such that whenever $J<T$, then there is an injective, non-surjective homomorphism $P_J \to P_T$. The \emph{fundamental group} of $G(\mathcal Q)$ is defined as the direct limit of the system $\{P_J\}_{J \in \mathcal Q}$. When $G(\mathcal Q)$ is strictly developable, the so-called basic construction provides an analogue of Davis complex which is called the \emph{standard development}. We propose a Bestvina complex for $G(\mathcal Q)$ and obtain the following result.

\begin{theorem}[Theorem~\ref{thm:main}]\label{thm:intromain} Let $G(\mathcal Q)$ be a strictly developable simple complex of finite groups over a poset $\mathcal{Q}$ with the fundamental group $G$. Then 
\begin{enumerate} 
\item the standard development $D(\D, G(\mathcal{Q}))$  and the Bestvina complex $D(\B, G(\mathcal{Q}))$ are $G$--homotopy equivalent,
\item \label{it:mainintro2} if $D(\D, G(\mathcal{Q}))$ is  a model for $\underline{E}G$ then $D(\B, G(\mathcal{Q}))$ is a cocompact model for $\underline{E}G$ 
satisfying
\[\mathrm{dim}(D(B, G(\mathcal{Q}))) = \left\{ \begin{array}{lcr}
      \ucd G		  & \text{ if } & \ucd G \neq 2,\\
      2 \text{ or } 3 & \text{ if } & \ucd G =2. \\

      \end{array} \right.\]
\end{enumerate}
\end{theorem}

To construct the complex $D(\B, G(\mathcal{Q}))$ it is enough to construct a compact polyhedron $\B$. We remark that the definition of $B$ depends only on the poset $Q$ and not, for example, on the subgroups $P_J$. Also, our procedure allows certain flexibility, which can often be used to obtain a complex with a simple cell structure. This, together with the minimal dimension of $D(\B, G(\mathcal{Q}))$ shows that $D(\B, G(\mathcal{Q}))$ is well-suited for cohomological computations.

Observe that when a simple complex of groups $G(\mathcal Q)$ is non-positively curved, then it is strictly developable and also $D(\D, G(\mathcal{Q}))$ becomes a model for $\underline{E}G$. So both parts of Theorem \ref{thm:intromain} apply in this case. In particular, since by a result of Moussong \cite[Theorem~11.1]{Davbuild} buildings are $\mathrm{CAT}(0)$, we obtain the following application of Theorem \ref{thm:intromain} to automorphism groups of buildings.


\begin{theorem}[Theorem~\ref{thm:building}]\label{thm:introbuilding}Let $G$ be a group acting properly and chamber transitively on a building of type $(W,S)$, and let  $G(\mQ)$ be  the associated simple complex of groups. Then $D(B, G(\mathcal{Q}))$ is a cocompact model for $\underline{E}G$ satisfying 

\[\mathrm{dim}(D(B, G(\mathcal{Q}))) = \left\{ \begin{array}{lcr}
      \vcd W        & \text{ if } & \vcd W \neq 2,\\
      2 \text{ or } 3 & \text{ if } & \vcd W=2. \\

      \end{array} \right.\]
\end{theorem}


 This result applies to finitely generated Coxeter groups and graph products of finite groups, since both are special cases of groups  acting properly and chamber transitively on a building. 

In \cite[Theorem~5.1]{DMP}, Degrijse and Mart\'\i nez-P\'erez give a general formula for computing the Bredon cohomological dimension of the fundamental group of $G(\mathcal Q)$. We simplify their formula and use it to compute the dimension of the Bestvina complex in Theorem \ref{thm:intromain}. In particular, Theorem~\ref{thm:introbuilding} can be viewed as the geometric counterpart of Theorem~5.4 of \cite{DMP}.  


Theorem \ref{thm:introbuilding} was also motivated by Remark~10.4 of Davis in \cite{Davbuild} where he hints at a possibility of such a construction.  We should also point out that Harlander and Meinert defined the Bestvina complex and obtained dimension bounds in \cite[Theorem~1.2]{HarMei} for graph products of finite groups.\smallskip

When a group $G$ acts properly on a finite dimensional contractible complex $X$ with suitable geometric properties, it is desirable to construct an equivariant deformation retraction of $X$ onto a subcomplex of minimal dimension, the so called `spine' of $X$.  Such spines have been constructed, for example, for certain arithmetic groups such as $\mbox{SL}(n, \Z)$ acting on the symmetric space \cite{Ash}, the outer automorphism groups of free groups acting on the Outer space \cite{Vog}, mapping class groups of punctured  surfaces acting on the Teichm\"{u}ller space \cite{Har} and others. We do not know whether in general the standard development equivariantly deformation retracts onto the Bestvina complex.  But when the Bredon cohomological dimension of the fundamental group of $G(\mathcal Q)$ is at most one, we obtain the following strengthening of Theorem~\ref{thm:intromain}. 

\begin{theorem}[Theorem~\ref{thm:virtfree}]\label{thm:introvirtfree}
Let $G(\mathcal Q)$ be a strictly developable simple complex of finite groups over the poset $\mathcal{Q}$ with the fundamental group $G$. Suppose that $D(\D, G(\mathcal{Q}))$ is  a model for $\underline{E}G$ and that $\ucd G\leqslant 1$. Then $D(K, G(\mathcal{Q}))$ equivariantly deformation retracts onto the tree $D(B, G(\mathcal{Q}))$.
\end{theorem}

As explained in Appendix~\ref{sec:appen}, the assumption that $\ucd G\leqslant 1$ can be weakened to $\ucd_R G\leqslant 1$ where $R$ is either a prime field or a subring of $\mathbb Q$ that contains $1$.  We point out that our proof of Theorem \ref{thm:introvirtfree} does not use the Accessibility Theory of groups and in particular it does not rely on Dunwoody's result \cite[Theorem~1.1]{Dun}. On the other hand, by combining Dunwoody's theorem with Theorem \ref{thm:introvirtfree}, the assumption $\ucd G \leqslant 1$ can be weakened to $\cd_{\mathbb{Q}}G \leqslant 1$.

In Section~\ref{sec:applications}, we discuss several classes of examples. We give explicit examples of Bestvina complex, which show that both the reduction of the dimension and the simplification of the cell structure can be substantial 
when compared with the standard development. We also give examples of groups for which $\ucd G=2$ but $\ugd G=3$, thus showing that the dichotomy in Theorem~\ref{thm:intromain}(\ref{it:mainintro2}) cannot be avoided.
 In fact these examples, first constructed in \cite{BLN},  give a negative answer to a question of Bestvina \cite[Remark~2]{Best}.
On the other hand, since we do not know an example of a group for which $\ugd G=2 $ but $\mathrm{dim}(D(\B, G(\mathcal{Q})))=3$, it is conceivable that one always has  $\mathrm{dim}(D(\B, G(\mathcal{Q}))=\ugd G$. In Appendix~\ref{sec:appen}, we give a construction of a Bestvina complex $B^R$ over a ring $R$ that is a subring of the rationals or a field of prime order and derive the analogue of Theorem \ref{thm:intromain} for the complex $B^R$.

We thank Pierre-Emmanuel Caprace and Ian Leary for helpful conversations.  


\section{Simple complexes of groups and the basic construction}\label{sec:scofgs}

In this section we recall the definitions of a simple complex of groups, a panel complex and the associated basic construction. In our exposition we follow \cite[II.12]{BH}, however, some definitions are adjusted to our purposes. We also prove two lemmas that describe the relationship between basic constructions coming from different panel complexes.\medskip

Throughout this section, let $\mathcal Q$ be a finite poset. 

\begin{definition}\label{def:scofg}A \emph{simple complex of finite groups} $G(\mathcal{Q}$) over $\mathcal Q$ consists of the following data:
\begin{enumerate}
\item \label{it:scofg1} for any $J \in \mathcal Q$ there is a finite group $P_J$, called the \emph{local group} at $J$,
	\item \label{it:scofg2} for any pair 
	$J < T$ there is an injective, non-surjective homomorphism \[\phi_{TJ} \colon P_J \to P_T,\] such that if $J < T < U$ then
  $\phi_{UT} \circ \phi_{TJ} = \phi_{UJ}$.
\end{enumerate}
\end{definition}
Given a simple complex of groups $G(\mathcal{Q}$) one defines its \emph{fundamental group} $\widehat{G(\mathcal{Q})}$ as the direct limit \[\widehat{G(\mathcal{Q})} = \varinjlim_{J \in \mathcal Q}P_J.\] For every $J \in \mathcal Q$ we have the canonical homomorphism $i_J \colon P_J \to \widehat{G(\mathcal{Q})}$. A simple complex of groups $G(\mathcal{Q})$ is called \emph{strictly developable} if every $i_J$ is injective.\medskip

From now on we assume that $G(\mathcal{Q})$ is strictly developable and let $G = \widehat{G(\mathcal{Q})}$. For any $J \in \mathcal{Q}$ we identify the group $P_J$ with its image $i_J(P_J)\subset G$.\medskip

We will now describe a procedure of constructing a space on which $G$ acts called the \emph{Basic construction}. First we need the following definition.

\begin{definition}\label{def:panelcx}A \emph{panel complex} $(X, \{\pan{X}{J}\}_{J \in \mathcal Q})$
 over a poset $Q$ is a compact polyhedron $X$ together with a family of subpolyhedra (called \emph{panels}) $\{\pan{X}{J}\}_{J \in \mathcal Q}$ such that the following conditions are satisfied:

\begin{enumerate}
\item polyhedron $X$ is the union of all panels,

\item if $J \leqslant T$ then $\pan{X}{T} \subseteq \pan{X}{J}$,

\item for any two panels 
their intersection is either empty or it is a union of panels.
\end{enumerate}
\end{definition}

\begin{definition}[Basic construction]
\label{def:basic}
Let $(X, \{\pan{X}{J}\}_{J \in \mathcal Q})$ be a panel complex over $Q$. Define the \emph{basic construction} $D(X, G(\mathcal{Q}))$ as
\[D(X, G(\mathcal{Q}))= (G \times X)/ \sim\]
such that $(g_1,x_1) \sim (g_2,x_2)$ if and only if $x_1=x_2$ and $g_1^{-1}g_2 \in P_{J(x_1)}$ 
 where $\pan{X}{J(x_1)}$ is the intersection of all the panels containing $x_1$. Let $[g,x]$ denote the equivalence class of $(g,x)$.
\end{definition}

There is a $G$--action on $D(X, G(\mathcal{Q}))$ given by $g \cdot [g',x] = [gg', x]$. Note that $D(X, G(\mathcal{Q}))$ has a natural structure of a polyhedral complex, and the $G$--action preserves that structure. Moreover, the stabilisers of this action are conjugates of the local groups $P_J$, and the quotient is homeomorphic to the panel complex $X$. In fact, $X$ is a \emph{strict fundamental domain} for this action, i.e.,\ if we view $X \cong [e,X]$ as a subcomplex of $D(X, G(\mathcal{Q}))$, we have that $X$ intersects every $G$--orbit at precisely one point. Since the local groups are finite and $X$ is compact, we conclude that the action of $G$ on $D(X, G(\mathcal{Q}))$ is proper and cocompact.

We would like to compare basic constructions arising from different panel complexes. For this we need some terminology. A \emph{panel map} between panel complexes $(X, \{\pan{X}{J}\}_{J\in \mathcal Q})$ and $(Y, \{\pan{Y}{J}\}_{J\in \mathcal Q})$ is a map $f \colon X \to Y$ such that for every $J \in \mathcal Q$ we have $f(\pan{X}{J}) \subseteq \pan{Y}{J}$. A \emph{panel homotopy} between panel maps $f_1 $ and $ f_2$ is a homotopy $H \colon X \times I \to Y$ between $f_1$ and $f_2$ such that for every $t \in I$ the restriction $H( -, t) \colon X \to Y$ is a panel map.

The following two lemmas appear to be elementary, however, to the best of our knowledge there are are no proofs of them in the literature. The special case of Coxeter systems is outlined in \cite[Proposition~11.5]{Davis}.

\begin{lemma}\label{lem:panelhtpy}Let $(Y, \{\pan{Y}{J}\}_{J\in \mathcal Q})$ be a panel complex over $\mathcal Q$ such that for every $J \in \mathcal Q$ the panel $\pan{Y}{J}$ is contractible. Then for any panel complex $(X, \{\pan{X}{J}\}_{J\in \mathcal Q})$ there is a panel map \[(X, \{\pan{X}{J}\}_{J\in \mathcal Q}) \to(Y, \{\pan{Y}{J}\}_{J\in \mathcal Q})\] which is unique up to panel homotopy.

In particular, any two panel complexes with contractible panels are panel homotopy equivalent. 
\end{lemma}

\begin{proof}First we prove the existence. 
 We will construct a family of maps $f_J \colon \pan{X}{J} \to \pan{Y}{J}$ such that if $J < J'$ then $f_{J}|_{J'}=f_{J'}$. After this is done, the required map $f \colon X \to Y$ can be defined as $f= \cup_{J \in \mathcal Q} f_J$.


For every maximal element $J \in \mathcal Q$ choose any map $ f_J \colon \pan{X}{J} \to \pan{Y}{J}$. Now given a panel $\pan{X}{J}$ such that for all $J'$ with $J < J'$ the map $f_{J'} \colon \pan{X}{J'} \to \pan{Y}{J'}$ has already been defined, we are searching for the following extension $f_J$:
\[
\begin{tikzcd}[column sep=5em]
 \flink{X}{J} \arrow{r}{\cup_{J < J'}f_{J'}} \arrow[hookrightarrow]{d} & \flink{Y}{J} \arrow[hookrightarrow]{d} \\
         \pan{X}{J}  \arrow[dashed]{r}{f_J} &  \pan{Y}{J}
\end{tikzcd}
\]
Since $\pan{Y}{J}$ is contractible, by \cite[Lemma~4.7]{Hat} the extension exists. This finishes the proof of the existence.

For the uniqueness, we proceed analogously. Given two panel maps $f, h \colon X \to Y$ let $f_J$ (resp.\ $h_J$) denote the restriction of $f $ (resp.\ $h$) to the panel $\pan{X}{J}$ and let $H_J \colon \pan{X}{J} \times I \to \pan{Y}{J}$ denote the homotopy between $f_J$ and $h_J$. This time, given $H_{J'}$ for every $J'$ with $J < J'$, we are looking for the extension $H_{J}$:
\[
\begin{tikzcd}[column sep=9em]
 (\pan{X}{J} \times \{0\}) \cup (\flink{X}{J} \times I) \cup (\pan{X}{J} \times \{1\}) \arrow{r}{f_{J} \,\cup \, (\cup_{J < J'} H_{J'}) \, \cup \, h_{J}} \arrow[hookrightarrow]{d} & \flink{Y}{J} \arrow[hookrightarrow]{d} \\
         \pan{X}{J} \times I  \arrow[dashed]{r}{H_J} &  \pan{Y}{J}
\end{tikzcd}
\]
As in the first case, the existence of such extension follows from contractibility of $\pan{Y}{J}$ and \cite[Lemma~4.7]{Hat}.\end{proof}

\begin{lemma}\label{lem:ghtpyeq}
Let $G(\mathcal Q)$ be a simple complex of finite groups over the poset $\mathcal{Q}$ and let $(X, \{\pan{X}{J}\}_{J\in \mathcal Q})$ and $(Y, \{\pan{Y}{J}\}_{J\in \mathcal Q })$ be two panel complexes over $\mathcal Q$. If $X$ and $Y$ are panel homotopy equivalent then the basic constructions $D(X, G(\mathcal{Q}))$ and $D(Y, G(\mathcal{Q}))$ are $G$--homotopy equivalent.
\end{lemma}

\begin{proof} Let $f \colon X \to Y$ and $h \colon Y \to X$ be the two panel maps such that $h \circ f$ (resp.\ $f \circ h$) is panel homotopic to the identity on $X$ (resp.\ $Y$). Define maps  $\tilde{f} \colon D(X, G(\mathcal{Q})) \to D(Y, G(\mathcal{Q}))$  and  $\tilde{h} \colon D(Y, G(\mathcal{Q})) \to D(X, G(\mathcal{Q}))$ by 
\[\tilde{f}([g,x])= [g,f(x)] \text{ and } \tilde{h}([g,x])= [g,h(x)].\]
One easily checks that both $\tilde{f}$ and $\tilde{h}$ are $G$--maps. 

Now let $H \colon X \times I \to X$ be the panel homotopy between $h \circ f$ and $\mathrm{id}_X$. Define the map $\widetilde{H} \colon D(X, G(\mathcal{Q})) \times I \to D(X, G(\mathcal{Q}))$ by 
\[\widetilde{H}([g,x],t)=[g,H(x,t)].\]
It is straightforward to check that $H$ is a $G$--homotopy between $\tilde{h}\circ \tilde{f}$ and $\mathrm{id}_{D(X, G(\mathcal{Q}))}$ A $G$--homotopy between $\tilde{f} \circ \tilde{h}$ and $\mathrm{id}_{D(Y, G(\mathcal{Q}))}$ is defined in the analogous way.
\end{proof}


\begin{remark}The notion of a panel complex appears in the literature under different names including \emph{stratified space} \cite{BH} or \emph{mirrored space} \cite{Davbook}. In our choice of terminology and notation we mostly follow \cite{DMP}. Consequently, any non-standard assumptions that we make (e.g.,\ a non-surjectivity assumption in Definition~\ref{def:scofg}.(\ref{it:scofg2})) are that of \cite{DMP}.\end{remark}

\section{Standard development and Bestvina complex}\label{sec:sdabc}

In this section we define two panel complexes over the poset $Q$: the `standard' complex $\D$ and the Bestvina complex $\B$. We compute the dimension of $\B$ and we give a simplification of the formula for $\underline{\mathrm{cd}}(G)$ where $G$ is the fundamental group of a simple complex of finite groups $G(\mathcal Q)$. After this is done, we prove Theorem~\ref{thm:main}.

As in the previous section, let $Q$ be a finite poset. There is a canonical panel complex associated to the poset $\mathcal Q$.

\begin{definition}\label{def:daviscx}Define the panel complex $(\D, \{\pan{\D}{J}\}_{J \in \mathcal Q})$ by \[\D= | \mathcal Q| \text{ and } \pan{\D}{J} = | \mathcal Q_{\geqslant J} |,\] where $\abs{-}$ denotes the geometric realisation of a poset, and $\mathcal Q_{\geqslant J}$ is a subposet of $\mathcal Q$ which consists of all elements that are greater or equal to $J$.
The basic construction $D(\D, G(\mathcal{Q}))$ will be referred to as the \emph{standard development} of the complex of groups $G(\mathcal Q)$.\end{definition}

Below we present another way to construct the complex $\D$ which will motivate the construction of the Bestvina complex.\medskip


For every maximal element $J \in \mathcal Q$ define $\pan{\D}{J}$ to be a point. Now given an element $J \in \mathcal Q$, suppose that for every $J'$ with $J < J'$ the panel $\pan{\D}{J'}$ has already been constructed. Define $\pan{\D}{J}$ to be the cone over the union $\flink{\D}{J}$ (the cone point corresponds to vertex $J$).\medskip

Now we define the Bestvina panel complex over the poset $Q$. 

\begin{definition}[Bestvina complex]\label{def:bestcx} The {\it Bestvina panel complex} $(\B, \{\pan{\B}{J}\}_{J \in \mathcal Q})$ is defined as follows. For every maximal element $J \in Q$ define $\pan{\B}{J}$ to be a point. Now given $J$ such that for all $J'$ with $J < J'$ the panel $\pan{\B}{J'}$ has already been constructed, 
define $\pan{\B}{J}$ to be a compact contractible polyhedron containing $\flink{\B}{J}$ of the smallest possible dimension.
\end{definition}

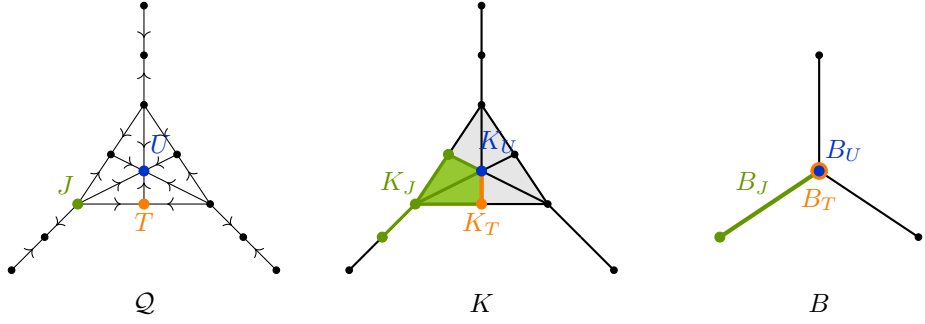
\begin{figure}[!h]
\centering
\begin{tikzpicture}[scale=0.44]

\definecolor{vlgray}{RGB}{230,230,230}
\definecolor{red}{RGB}{100,150,00}
 \definecolor{tgreen}{RGB}{250,125,000}
\definecolor{lred}{RGB}{150,200,050}

 \definecolor{blue}{RGB}{000,50,200}


\begin{scope}[shift={(-10.2,0)}]
\draw[ ->-] (-2,0) to (0,0);
\draw[->-] (-2,0) to (0,1);
\draw[ ->-] (-2,0) to (-1,1.5);
\draw[->-] (-2,0) to (-3,-1);

\draw[ ->-] (2,0) to (0,0);
\draw[->-] (2,0) to (0,1);
\draw[->-] (2,0) to (1,1.5);
\draw[ ->-] (2,0) to (3,-1);

\draw[ ->-] (0,3) to (0,4.5);

\draw[->-] (0,3) to (0,1);

\draw[->-] (0,3) to (-1,1.5);
\draw[->-] (0,3) to (1,1.5);

\draw[->-] (-4,-2) to (-3,-1);
\draw[->-] (4,-2) to (3,-1);
\draw[->-] (0,6) to (0,4.5);

\draw[->-] (0,0) to (0,1);
\draw[->-] (-1,1.5) to (0,1);
\draw[->-] (1,1.5) to (0,1);

\draw[fill] (-2,0)   circle [radius=0.1];

\draw[fill] (0,0)   circle [radius=0.1];

\draw[fill] (2,0)   circle [radius=0.1];

\draw[fill] (3,-1)   circle [radius=0.1];

\draw[fill] (4,-2)   circle [radius=0.1];

\draw[fill] (-3,-1)   circle [radius=0.1];

\draw[fill] (-4,-2)   circle [radius=0.1];

\draw[fill] (0,1)   circle [radius=0.1];

\draw[fill] (-1,1.5)   circle [radius=0.1];

\draw[fill] (1,1.5)   circle [radius=0.1];

\draw[fill] (0,3)   circle [radius=0.1];

\draw[fill] (0,4.5)   circle [radius=0.1];

\draw[fill] (0,6)   circle [radius=0.1];

\draw[fill=red, red] (-2,0)   circle [radius=0.15];

\draw[fill=tgreen,tgreen] (0,0)   circle [radius=0.15];

\draw[fill=blue, blue] (0,1)   circle [radius=0.15];


\node   at (0,-3)  {$\mathcal{Q}$};

\node [below ,red]   at (-2.375,1.125)  {$J$};
\node [below, tgreen]   at (0,0)  {$T$};
\node [above, blue]   at (0.5,1.25)  {$U$};

\end{scope}


\draw[fill=vlgray] (-2,0)-- (2,0)--(0,3)--(-2,0);

\draw[fill] (2,0)   circle [radius=0.1];

\draw[fill] (4,-2)   circle [radius=0.1];

\draw[fill] (-3,-1)   circle [radius=0.1];

\draw[fill] (-4,-2)   circle [radius=0.1];

\draw[fill] (0,1)   circle [radius=0.1];

\draw[fill] (-1,1.5)   circle [radius=0.1];

\draw[fill] (1,1.5)   circle [radius=0.1];

\draw[fill] (0,3)   circle [radius=0.1];

\draw[fill] (0,4.5)   circle [radius=0.1];

\draw[fill] (0,6)   circle [radius=0.1];

\draw[thick] (2,0) to (0,0);
\draw[thick] (2,0) to (0,1);
\draw[thick] (2,0) to (1,1.5);
\draw[thick] (2,0) to (3,-1);

\draw[thick] (0,3) to (0,4.5);

\draw[thick] (0,3) to (0,1);

\draw[thick] (0,3) to (-1,1.5);
\draw[thick] (0,3) to (1,1.5);

\draw[thick] (-4,-2) to (-3,-1);
\draw[thick] (4,-2) to (3,-1);
\draw[thick] (0,6) to (0,4.5);

\draw[thick] (1,1.5) to (0,1);


\draw[very thick,red, fill=lred] (-2,0)-- (0,0)  --(0,1)--(-2,0);
\draw[very thick, red] (-2,0) to (0,1);

\draw[very thick ,red] (-2,0) to (-1,1.5);
\draw[very thick, red] (-2,0) to (-3,-1);
\draw[very thick, red] (0,0) to (0,1);
\draw[ very thick, red,  fill=lred] (0,0)-- (0,1) -- (-1, 1.5) --(-2,0)--(0,0);
\draw[very thick, red] (-2,0) to (0,1);

\draw[red,fill=red] (-2,0)   circle [radius=0.15];

\draw[fill=red, red] (-1,1.5)   circle [radius=0.15];

\draw[fill=red, red] (-3,-1)   circle [radius=0.15];

\draw[ultra thick, tgreen] (0,0) to (0,1);
\draw[fill=tgreen,tgreen] (0,0)   circle [radius=0.15];

\draw[fill=blue,blue] (0,1)   circle [radius=0.15];

\node   at (0,-3)  {$\D$};

\node [below ,red]   at (-2.5,1.25)  {$\D_J$};
\node [below, tgreen]   at (0,0)  {$\D_T$};

\node [above, blue]   at (0.5,1.25)  {$\D_U$};



\begin{scope}[shift={(10.2,0)}]
\draw[thick] (3,-1) to (0,1);
\draw[thick] (-3,-1) to (0,1);
\draw[thick] (0, 4.5) to (0,1);

\draw[fill] (3,-1)   circle [radius=0.1];

\draw[fill] (-3,-1)   circle [radius=0.1];

\draw[fill] (0,1)   circle [radius=0.1];

\draw[fill] (0,4.5)   circle [radius=0.1];

\draw[ultra thick, red] (-3,-1) to (0,1);
\draw[fill=red, red] (-3,-1)   circle [radius=0.15];

\draw[fill=tgreen, tgreen] (0,1)   circle [radius=0.25];
\draw[fill=blue, blue] (0,1)   circle [radius=0.15];

\node   at (0,-3)  {$\B$};

\node [below ,red]   at (-2,1.25)  {$\pan{\B}{J}$};
\node [below, tgreen]   at (0,0.75)  {$\pan{\B}{T}$};
\node [above, blue]   at (0.75,1)  {$\pan{\B}{U}$};

\end{scope}


\end{tikzpicture}
\caption{Poset $\mathcal{Q}$ and complexes $\D$ and $\B$. The panels of $\D$ and $\B$ corresponding to elements 
 $J<T<U$ are respectively green, orange and blue. Observe that we have $\pan{\B}{T}=\pan{\B}{U}$.}
\label{fig:posetandcomplexes}
\end{figure}

Note that for any compact polyhedron $L$, the cone $C(L)$ is a compact contractible polyhedron containing $L$. However, $\mathrm{dim}(C(L))=\mathrm{dim}(L)+1 $. The following lemma gives a sufficient condition for the existence of the contractible polyhedron that has the same dimension as $L$. It was first proved by Bestvina  \cite[Lemma p.~21]{Best}. We include the proof for the sake of completeness and also because our assumptions are slightly more general. An even more general version (which we use in Appendix~\ref{sec:appen}) is proved in \cite[Lemma~24]{LeSa}.\medskip

\noindent
Unless stated otherwise, all (co)homology groups are taken with  coefficients in $\mathbb{Z}$.
\begin{lemma}\label{lem:polyhedron}Let $L$ be a compact polyhedron of dimension $n \neq 2$. If $\rH^n(L)=0$ then $L$ embeds into a contractible compact polyhedron of dimension $n$.  
\end{lemma}
\begin{proof}
Note that if  $n \leqslant 1$ then there is nothing to prove, thus we can assume that $n \geqslant 3$. Let $C(L^{(n-2)})$ be the cone on the $(n-2)$--skeleton of $L$. By replacing $L$ with $L \cup C(L^{(n-2)})$ if necessary, we can assume that $L$ is $(n-2)$--connected.
 By the Universal Coefficients Theorem we have
\[\rH^n(L) \cong \mathrm{Hom}(\rH_n(L), \mathbb{Z}) \oplus \mathrm{Ext}(\rH_{n-1}(L), \mathbb{Z}).\]
We conclude that $\rH_n(L)$ is torsion, and hence it must be trivial (because $\mathrm{dim}(L)=n$ and thus $\rH^n(L)=\mathrm{ker}(d_n)$ is torsion-free).
We also obtain that $\rH_{n-1}(L)$ is a finitely generated free abelian group. 
Since $L$ is $(n-2)$--connected, by the Hurewicz Theorem we have $\pi_{n-1}(L) \cong \rH_{n-1}(L)$. 
For each generator $[f]$ of $\pi_{n-1}(L)$ attach to $L$ an $n$--cell along the map $f \colon S^{n-1} \to L$ and call the resulting space $L'$. We can choose each $f$ to be a PL-map and thus $L'$ is a polyhedron. One easily checks that $L'$ is contractible.
\end{proof}

\begin{proposition}\label{prop:dimbest}Let $d$ be an integer defined as 
\begin{equation}\label{eq:dimbestv}d= \mathrm{max}\{n \in \mathbb{N} \mid \rH^{n-1}\big(\flink{\B}{J}\big) \neq 0 \text{ for some } J \in \mathcal Q \}.\end{equation}
Then the dimension of Bestvina complex is given by
\[\mathrm{dim}(\B) = \left\{ \begin{array}{lcr}

      d 			  & \text{ if } & d \neq 2,\\
      2 \text{ or } 3 & \text{ if } & d =2. \\

      \end{array} \right.\]

\end{proposition}
\begin{proof}
First note that for any $J \in \mathcal Q$ the polyhedron $\flink{\B}{J}$ is contained in the contractible polyhedron $\pan{\B}{J}$. Thus if $\rH^{n-1}\big(\flink{\B}{J}) \neq 0$ for some $n>0$ then $\pan{\B}{J}$ necessarily has dimension at least $n$. This shows that $\mathrm{dim}(B) \geqslant d$. 

Now assume that $\mathrm{dim}(\B)=k \geqslant 0$ and 
let $J$ be the last element in the construction of $\B$ for which $k=\mathrm{dim}(\pan{\B}{J}) > \mathrm{dim}(\flink{\B}{J})$. In light of Lemma~\ref{lem:polyhedron}, this means that either $\rH^{k-1}(\flink{\B}{J}) \neq 0$ or $\mathrm{dim}(\flink{\B}{J})=2$. In the first case, we obtain that $k=d$. In the second case, we necessarily have $\rH^{1}(\flink{B}{T}) \neq 0$ for some $T$ which appears in the construction earlier than $J$, for otherwise $\flink{\B}{J}$ would not have dimension $2$ to begin with. Consequently, we get that $d \in \{2,3\}$ and $k=3$. \end{proof}

\begin{remark}\label{rem:2dimacyclic}In contrast to Lemma~\ref{lem:polyhedron}, if $L$ is a $2$--dimensional acyclic polyhedron with a finite, non-trivial fundamental group then $L$ cannot be embedded into a $2$--dimensional contractible polyhedron \cite[ Proposition~5]{BLN}. Thus in Proposition~\ref{prop:dimbest} we cannot avoid the possibility that $d=2$ and $\mathrm{dim}(B)=3$.\end{remark}

Now let $G(\mathcal Q)$ be a strictly developable simple complex of finite groups with the fundamental group $G$ and suppose that the standard development $D(\D,G(\mathcal Q))$ is a model for $\underline{E}G$.

In \cite{DMP}, Degrijse and Mart\'\i nez-P\'erez  give a formula for the Bredon cohomological dimension of $G$. We will now show how to simplify their formula, in order to compare it with \eqref{eq:dimbestv}.   
We need the following definitions. For a subset $\Omega \subseteq \mathcal{Q}$ define subspaces ${\D}_{\Omega}$ and ${\D}_{>\Omega}$ of $\D$ to be the following geometric realisations
\begin{align*}K_{\Omega}= \abs{ \{ V \in \mathcal Q \mid V \geqslant U \text{ for some } U \in \Omega\}},\\
K_{>\Omega}= \abs{ \{ V \in \mathcal Q \mid V > U \text{ for some } U \in \Omega\}}.
\end{align*}
Note that $K_{\{U\}}$ is the panel $\pan{K}{U}$ and $K_{> \{U\}}$ is the union $\flink{K}{U}$. We will abbreviate $K_{\{U\}}$ to $\pan{K}{U}$ and $K_{>\{U\}}$ to $K_{>U}$.\medskip

For $J \in \mathcal Q$, define the subset $\Omega_J \subseteq \mathcal Q$ as \[\Omega_J= \{U \in \mathcal Q \mid P_J=P_U\},\]
where $P_J$ and $P_U$ are the local groups corresponding to $J$ and $U$ respectively. By \cite[Theorem 5.1]{DMP}, one has
\begin{equation}\label{eq:dimdieter}\underline{\mathrm{cd}}(G)= \mathrm{max}\{ n \in \mathbb{N} \mid \rH^n(K_{\Omega_J},K_{>\Omega_J}) \neq 0 \text{ for some } J \in \mathcal Q\}.\end{equation} 

\begin{proposition}\label{prop:bredondim}Under the above assumptions, we have 
\begin{equation}\label{eq:dimdavis}\underline{\mathrm{cd}}(G)=\mathrm{max}\{n \in \mathbb{N} \mid \rH^{n-1}(K_{>J}) \neq 0 \text{ for some } J \in \mathcal Q \}.\end{equation}
\end{proposition}

\begin{proof}We will show that for every $J \in \mathcal Q$ and for any integer $n\geqslant 0$ we have \begin{equation}\label{eq:dimensioncompare}\rH^n(K_{\Omega_J},K_{>\Omega_J})\cong \bigoplus_{U \in \Omega_J}\rH^n(K_{U},K_{>U}).\end{equation}
This will imply the proposition, since for any $U \in \mathcal Q$, the space $K_{U}$ is contractible and thus $\rH^n(K_{U},K_{>U}) \cong \rH^{n-1}(K_{>U})$. To show \eqref{eq:dimensioncompare}, we proceed by induction on the number of elements in $\Omega_J$. If $\Omega_J$ contains only one element then \eqref{eq:dimensioncompare} is clearly satisfied. Assume now that $\Omega_J$ contains more than one element. Let $U \in \Omega_J$, let $\omminu= \Omega_J \smallsetminus \{U\}$ and write the pair $(K_{\Omega_J}, K_{> \Omega_J})$ as
\[(K_{\Omega_J}, K_{> \Omega_J})= (K_{\omminu} \cup K_{U}, K_{> \omminu} \cup K_{>U}).\]
The relative Mayer-Vietoris sequence reads
\begin{align*} \rH^{n-1}(K_{U}  \cap K_{ \omminu},K_{> U} \cap K_{> \omminu} )  \to & \\ 
 \to \rH^n(K_{\Omega_J}, K_{> \Omega_J})\to \rH^n  (K_{U}&, K_{> U}) \oplus  \rH^n(K_{\omminu}, K_{> \omminu})  \to \\
 & \to \rH^{n}(K_{U} \cap K_{\omminu},K_{> U} \cap K_{> \omminu} ).\end{align*} 
\begin{claim} We have $K_{U}  \cap K_{\omminu} =K_{> U}  \cap K_{> \omminu}$.\end{claim}
To prove the claim consider an element $V \in K_{U}  \cap K_{\omminu}$. Thus $U \leqslant V$ and $U' \leqslant V$ for some $U' \in \omminu$. Suppose that $U=V$. In this case the map $P_{U'} \to P_U$ is an isomorphism since it is injective and $P_{U'}=P_{U}$ by definition of $\Omega_J$. This contradicts the assumption that no homomorphism between local groups is surjective (see Definition~\ref{def:scofg}(\ref{it:scofg2})). For the same reason we cannot have $V=U'$ for any $U' \in \omminu$. This finishes the proof of the claim.
 
 The claim implies that $\rH^{n}(K_{U}  \cap K_{\omminu},K_{> U} \cap K_{> \omminu} )=0$ for every $n\geqslant 0$ and therefore the map
\[\rH^n(K_{\Omega_J}, K_{> \Omega_J})\to \rH^n  (K_{U}, K_{> U}) \oplus  \rH^n(K_{\omminu}, K_{> \omminu})\] 
is an isomorphism. Since by the inductive assumption we have $\rH^n(K_{\omminu},K_{>\omminu})\cong \bigoplus_{U' \in \omminu}\rH^n(K_{U'},K_{>U'})$, the formula \eqref{eq:dimensioncompare} is established.
\end{proof}

We are ready now to prove the main theorem.

\begin{theorem}\label{thm:main} Let $G(\mathcal Q)$ be a strictly developable simple complex of finite groups over a poset $\mathcal{Q}$ with the fundamental group $G$. Then 
\begin{enumerate} 
\item \label{it:main1} the standard development $D(\D, G(\mathcal{Q}))$ and the Bestvina complex $D(\B, G(\mathcal{Q}))$ are $G$--homotopy equivalent,
\item \label{it:main2} if $D(\D, G(\mathcal{Q}))$ is  a model for $\underline{E}G$ then $D(\B, G(\mathcal{Q}))$ is a cocompact model for $\underline{E}G$ 
satisfying
\[\mathrm{dim}(D(\B, G(\mathcal{Q}))) = \left\{ \begin{array}{lcr}
      \underline{\mathrm{cd}}(G) 			  & \text{ if } & \underline{\mathrm{cd}}(G) \neq 2,\\
      2 \text{ or } 3 & \text{ if } & \underline{\mathrm{cd}}(G) =2. \\

      \end{array} \right.\]
\end{enumerate}
\end{theorem}

\begin{proof}
(\ref{it:main1}). By definition both $\D$ and $\B$ have contractible panels and thus by Lemma~\ref{lem:panelhtpy} they are panel homotopy equivalent. Lemma~\ref{lem:ghtpyeq} implies then that $D(\D, G(\mathcal{Q}))$ and $D(\B, G(\mathcal{Q}))$ are $G$--homotopy equivalent.

(\ref{it:main2}). By (\ref{it:main1}) we have that $D(\D, G(\mathcal{Q}))$ is a model for $\underline{E}G$ if and only if $D(\B, G(\mathcal{Q}))$ is so. Since the quotient $D(\B, G(\mathcal{Q}))/G \cong \B$ is compact, we get that $D(\B, G(\mathcal{Q}))$ is a cocompact model for $\underline{E}G$.
It remains to show that $\mathrm{dim}(D(B, G(\mathcal{Q}))) =  \underline{\mathrm{cd}}(G)$ if $\underline{\mathrm{cd}}(G)  \neq 2 $ and that
$\mathrm{dim}(D(\B, G(\mathcal{Q}))) \in  \{2,3\}$ if $\underline{\mathrm{cd}}(G)=2$.

First note that clearly $\mathrm{dim}(D(\B, G(\mathcal{Q}))) =\mathrm{dim}(\B)$. Thus in light of Propositions~\ref{prop:dimbest} and \ref{prop:bredondim} it suffices to show that formulas \eqref{eq:dimbestv} and \eqref{eq:dimdavis} agree. We will show that for any $J \in Q$ and any $n > 0$ we have 
\[\rH^{n}\big(\flink{\B}{J}\big) \cong \rH^{n}\big(\flink{\D}{J}\big)\]
Recall that $\D_{>J} = \flink{\D}{J}$ (see the discussion before Proposition~\ref{prop:bredondim}). Note that both $\flink{\D}{J}$ and $\flink{\B}{J}$ may be seen as panel complexes over the poset $\mathcal Q_{>J}$, by restricting the panel structure from $\D$ and $\B$ respectively. By Lemma~\ref{lem:panelhtpy} we conclude that $\flink{\D}{J}$ and $\flink{B}{J}$ are (panel) homotopy equivalent. In particular, their cohomology groups are isomorphic.
\end{proof}

\begin{remark}Note that in the above theorem we do not assume that $\D$ (or $\B$) is contractible. However, if $D(\D, G(\mathcal{Q}))$ is contractible (which is the case for example when it is a model for $\underline{E}G$) then $\D$ is contractible as well, since it is a retract of $D(\D, G(\mathcal{Q}))$ (see the discussion after Definition~\ref{def:basic}).
\end{remark}

\section{Equivariant deformation retraction}\label{sec:eqdefret}

In this section we discuss when it is possible to obtain the basic construction $D(\B, G({\mathcal Q}))$ as an equivariant deformation retract of $D(\D, G({\mathcal Q}))$. We isolate a concrete condition which ensures that this is the case. 
We show that this condition is always satisfied if the Bredon cohomological dimension of the fundamental group $G$ is at most one.\medskip

Let $(X, \{\pan{X}{J}\}_{J \in \mathcal Q})$ and $(Y, \{\pan{Y}{J}\}_{J \in \mathcal Q})$ be panel complexes over $\mathcal Q$. A \emph{panel inclusion} $ i \colon X \hookrightarrow Y$ is an injective panel map. A \emph{panel deformation retraction} $ r \colon Y \to X$ is panel map such that $r\circ i$ is panel homotopic to $\mathrm{id}_{\D}$ rel $i(X)$.

The following lemma is similar to Lemma~\ref{lem:panelhtpy}.

\begin{lemma}\label{lem:paneldr}
Suppose $(X, \{\pan{X}{J}\}_{J \in \mathcal Q})$ and $(Y, \{\pan{Y}{J}\}_{J \in \mathcal Q})$ are panel complexes over $\mathcal Q$ and let $i \colon X \hookrightarrow Y$ be a map which is a panel inclusion and such that for every $J \in \mathcal Q$ the restriction $i_J \colon \pan{X}{J} \hookrightarrow \pan{Y}{J}$ is a homotopy equivalence. Then there is a panel deformation retraction $r \colon Y \to X$. 
\end{lemma}
\begin{proof}The proof follows the same idea as the proof of Lemma~\ref{lem:panelhtpy}. For $J \in \mathcal{Q}$ 
let $r_J$ denote the restriction of the putative map $r$ to the panel $ \pan{Y}{J}$. We will construct the maps $r_J$ inductively and then set $r= \cup_{J \in \mathcal Q} r_J$.  

For every maximal element $J$ we have an inclusion $i_J \colon \pan{X}{J} \hookrightarrow  \pan{Y}{J}$ which is a homotopy equivalence. Then by \cite[Corollary 0.19]{Hat} there exists a deformation retraction $r_J \colon \pan{Y}{J} \to \pan{X}{J}$. For the inductive step, suppose that $J \in \mathcal Q$ is such that for any $J'$ with $J < J'$ the map $r_{J'}$ has been defined and that we have a panel inclusion 
\[\cup_{J\leqslant J'}  i_{J'} \colon \flink{X}{J} \hookrightarrow \flink{Y}{J}\] and a panel deformation retraction
\[\cup_{J\leqslant J'}  r_{J'} \colon \flink{Y}{J} \to \flink{X}{J}.\]
We will construct a deformation retraction $r_J \colon \pan{Y}{J} \to \pan{X}{J}$ such that the following diagram commutes
\[
\begin{tikzcd}[column sep=5em, row sep =4em]
 \flink{X}{J} \arrow[r, hook, shift left=-0.75ex, "\cup_{J < J'} \ i_{J'}"' ]
  \arrow[hookrightarrow]{d} & \flink{Y}{J}  \arrow[l,  shift right=0.75ex, "\cup_{J < J'} \ r_{J'}"' ]   \arrow[hookrightarrow]{d} \\
         \pan{X}{J}  \arrow[r, hook, shift left=-0.75ex, "i_J"' ] &  \pan{Y}{J}  \arrow[l, dashed,  shift right=0.75ex, "r_J"' ]
\end{tikzcd}
\]
and such that the homotopy between $r_J\circ i_J$ and $\mathrm{id}_{\pan{Y}{J}}$ restricts to the homotopy between $(\cup_{J < J'}  r_{J'})\circ(\cup_{J < J'}  i_{J'})$ and $\mathrm{id}_{(\flink{Y}{J})}$.

To construct $r_J$, first consider the pushout \[A =\pan{X}{J}\cup_{(\flink{X}{J})}(\flink{Y}{J}).\] Note that there is a deformation retraction $ p_A \colon A \to \pan{X}{J}$ given by performing the identity on $X_J$ and $\cup_{J < J'}  r_{J'}$ on $\flink{Y}{J}$. 
Then consider the inclusion $A \hookrightarrow \pan{Y}{J}$. One easily verifies that this inclusion is a homotopy equivalence. Thus by \cite[Corollary~0.20]{Hat} we get a deformation retraction $ p_{Y_J} \colon Y_J \to A$. Define $r_J =p_{A}\circ p_{Y_J}$. It is straightforward to check that $r_J$ has all the claimed properties. This finishes the inductive step. Now setting $r= \cup_{J \in \mathcal Q} r_J$ gives the desired panel deformation retraction.
\end{proof}

\begin{lemma}\label{lem:groupdr}Suppose that we have a panel inclusion $i \colon X \hookrightarrow Y$ and a panel deformation retraction $r \colon Y \to X$. Then there is a $G$--equivariant inclusion \[\tilde{i} \colon D(X, G({\mathcal Q})) \hookrightarrow D(Y, G({\mathcal Q}))\] and a $G$--equivariant deformation retraction \[\tilde{r} \colon D(Y, G({\mathcal Q})) \to D(X, G({\mathcal Q})).\] 
\end{lemma}

\begin{proof}The proof is the same as the proof of Lemma~\ref{lem:ghtpyeq}.
\end{proof}

We will now state the condition which will ensure the existence of an equivariant deformation retraction $D(\D, G(\mathcal{Q})) \to D(\B, G(\mathcal{Q}))$. We do not know whether it is always satisfied. 


\begin{definition}\label{cond:1}Let $L$ be a compact polyhedron of dimension $n \neq 2$. We say that $L$ is {\it subconical} if either $\rH^n(L)\ne 0$ or  when $\rH^n(L)=0$, then $L$ embeds into an $n$--dimensional contractible compact polyhedron $L'$, such that $L'$ is a subpolyhedron of the (possibly subdivided) cone $C(L)$ and the composite $L \hookrightarrow L' \hookrightarrow C(L)$ is equal to the canonical inclusion $L \hookrightarrow C(L)$.
\end{definition}

\begin{ques}\label{ques_cone} Is every compact polyhedron of dimension greater than two subconical?
\end{ques}

\begin{remark}Note that in Lemma~\ref{lem:polyhedron}, the construction of a contractible polyhedron containing $L$  begins with attaching the cone $C(L^{(n-2)})$ to $L$.  Next, one has to attach $n$--cells in order to kill the generators of $\pi_{n-1}(L \cup C(L^{(n-2)})) \cong \mathbb{Z}^k$. We do not know whether this step can be performed so that the resulting polyhedron becomes subconical. 
\end{remark}



\begin{lemma}\label{lem:conditionsatisf}Suppose that in the construction of Bestvina complex, at every step the polyhedron $\flink{\B}{J}$ is subconical. Then there is a panel inclusion $\B \hookrightarrow \D$ such that for any $J \in \mathcal Q$ the restriction $i_J \colon \pan{\B}{J} \hookrightarrow \pan{\D}{J}$ is a homotopy equivalence.
\end{lemma}

\begin{proof}For any maximal element $J \in \mathcal Q$ both panels $\pan{\B}{J}$ and $\pan{\D}{J}$ are equal to the point, so we set $i_J \colon \pan{\B}{J} \hookrightarrow \pan{\D}{J}$ to be the identity map. Now given $J \in \mathcal Q$, assume inductively that we have an inclusion  \[\cup_{J < J'}  i_{J'} \colon \flink{\B}{J} \hookrightarrow \flink{\D}{J}.\] Since $\flink{\B}{J}$ is subconical, we can choose $B_J$ to be the subcone of $C(\flink{\B}{J})$. Define $i_J$ to be the composition \[\pan{\B}{J} \hookrightarrow C(\flink{\B}{J}) \hookrightarrow C(\flink{\D}{J}) =\pan{\D}{J},\] where the latter map is the cone on $\cup_{J < J'}  i_{J'}$. By construction $i_J$ restricts to $\cup_{J < J'}  i_{J'}$ over $\flink{\B}{J}$. Finally, since both $\pan{\B}{J} $ and $\pan{\D}{J}$ are contractible, we conclude that $i_J$ is a homotopy equivalence.
\end{proof}



By Lemmas~\ref{lem:paneldr}, \ref{lem:groupdr} and \ref{lem:conditionsatisf}, we obtain the following.

\begin{proposition}\label{prop:defretract}Assume that for any $J \in Q$ in the construction of Bestvina complex, the subpolyhedron $\flink{\B}{J}$ is subconical. Then 
$D(\B, G(\mathcal{Q}))$ is a $G$--equivariant deformation retract of 
$D(\D, G(\mathcal{Q}))$.
\end{proposition}

We conclude this section with an application of the above proposition in the case of groups of Bredon cohomological dimension at most one. 

\begin{theorem}\label{thm:virtfree}
Let $G(\mathcal Q)$ be a strictly developable simple complex of finite groups over the poset $\mathcal{Q}$ with the fundamental group $G$. Suppose that $D(\D, G(\mathcal{Q}))$ is  a model for $\underline{E}G$ and that $\ucd  G\leqslant 1$. Then $D(\D, G(\mathcal{Q}))$ equivariantly deformation retracts onto the tree $D(\B, G(\mathcal{Q}))$.
\end{theorem}

\begin{proof}We will first prove that $D(\B, G(\mathcal{Q}))$ is a $G$--equivariant deformation retract of $D(\D, G(\mathcal{Q}))$. For this, in light of Proposition~\ref{prop:defretract}, it is enough to show that in the construction of Bestvina complex, for any $J \in \mathcal Q$ the polyhedron $\flink{\B}{J}$ is subconical. 

First note that the condition does not apply to maximal elements of $\mathcal Q$.
Now suppose that $J \in \mathcal Q$ is such that all the panels in $\flink{\B}{J}$ have been defined. Since $\ucd G \leqslant 1$, by Proposition~\ref{prop:bredondim} we have that $\rH^{n}(\flink{\D}{J}) = 0$ for all $n \geqslant 1$. Thus every connected component of $\flink{\D}{J}$ is contractible. Since $\flink{\B}{J}$ and $\flink{\D}{J}$ are homotopy equivalent (cf.\ proof of Theorem~\ref{thm:main}.(\ref{it:main2})), the same is true for $\flink{\B}{J}$. 

 Take $\flink{\B}{J}$ and add a disjoint vertex $v_J$. Pick a vertex $v_i$ in every connected component of $\flink{\B}{J}$ and join every $v_i$ with $v_J$ by an edge. One easily verifies that the resulting space is a contractible polyhedron which embeds into $C(\flink{\B}{J})$ (vertex $v_J$ is sent to the cone-point of $C(\flink{\B}{J})$). 
Thus this polyhedron is subconical. 

It remains to show that $D(\B, G(\mathcal{Q}))$ is a tree. By the above $D(\B, G(\mathcal{Q}))$ is a $G$--deformation retract of $D(\D, G(\mathcal{Q}))$. It follows from Theorem~\ref{thm:main} and the assumption that $\ucd G\leqslant 1$ that $D(\B, G(\mathcal{Q}))$ is at most a $1$--dimensional model for $\ue G$, hence a tree.
\end{proof}

\begin{figure}[!h]
\centering
\begin{tikzpicture}[scale=0.44]

\definecolor{vlgray}{RGB}{230,230,230}
\definecolor{tpurple}{RGB}{250,125,000}


\draw[fill=vlgray] (-2,0)-- (2,0)--(0,3)--(-2,0);

\draw[fill] (2,0)   circle [radius=0.1];

\draw[fill] (0,0)   circle [radius=0.1];

\draw[fill] (-2,0)   circle [radius=0.1];

\draw[fill] (4,-2)   circle [radius=0.1];

\draw[fill] (-3,-1)   circle [radius=0.1];
\draw[fill] (3,-1)   circle [radius=0.1];

\draw[fill] (-4,-2)   circle [radius=0.1];

\draw[fill] (0,1)   circle [radius=0.1];

\draw[fill] (-1,1.5)   circle [radius=0.1];

\draw[fill] (1,1.5)   circle [radius=0.1];

\draw[fill] (0,3)   circle [radius=0.1];

\draw[fill] (0,4.5)   circle [radius=0.1];

\draw[fill] (0,6)   circle [radius=0.1];

\draw[thick] (2,0) to (0,0);
\draw[thick] (2,0) to (0,1);
\draw[thick] (2,0) to (1,1.5);
\draw[thick] (2,0) to (3,-1);

\draw[thick] (-2,0) to (0,0);
\draw[thick] (-2,0) to (0,1);
\draw[thick] (-2,0) to (-1,1.5);
\draw[thick] (-2,0) to (-3,-1);

\draw[thick] (0,3) to (0,4.5);

\draw[thick] (0,3) to (0,1);

\draw[thick] (0,3) to (-1,1.5);
\draw[thick] (0,3) to (1,1.5);

\draw[thick] (-4,-2) to (-3,-1);
\draw[thick] (4,-2) to (3,-1);
\draw[thick] (0,6) to (0,4.5);

\draw[thick] (1,1.5) to (0,1);
\draw[thick] (-1,1.5) to (0,1);
\draw[thick] (0,0) to (0,1);



\draw[fill=tpurple, tpurple] (0,1)   circle [radius=0.1];
\draw[fill=tpurple, tpurple] (0,4.5)   circle [radius=0.1];
\draw[fill=tpurple, tpurple] (-3,-1)   circle [radius=0.1];
\draw[fill=tpurple, tpurple] (3,-1)   circle [radius=0.1];
\draw[fill=tpurple, tpurple] (0,3)   circle [radius=0.1];
\draw[fill=tpurple, tpurple] (2,0)   circle [radius=0.1];
\draw[fill=tpurple, tpurple] (-2,0)   circle [radius=0.1];

\draw[thick, tpurple] (2,0) to (0,1);
\draw[thick, tpurple] (-2,0) to (0,1);
\draw[thick, tpurple] (0,4.5) to (0,1);
\draw[thick, tpurple] (-3,-1) to (-2,0);
\draw[thick, tpurple] (3,-1) to (2,0);

\node [tpurple]   at (1,3.75)  {$\B$};
\node [black]   at (2,2)  {$\D$};

\node [below right]   at (-2,0)  {$J$};

\node [tpurple]   at (-2.5,0.5)  {$v_J$};

\end{tikzpicture}
\caption{Complex $\B$ (orange) embeds into $\D$ (black) as a panel deformation retract. After embedding, vertex $v_J \in \pan{\B}{J}$ is equal to vertex $J \in\pan{\D}{J}$.}
\label{fig:defretract}
\end{figure}
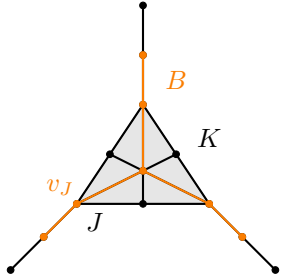

\section{Applications and examples}\label{sec:applications}

In this section we present some classes of groups to which our results apply. In particular, we give a proof of Theorem~\ref{thm:building}. 
We also give explicit examples of Bestvina complex in some cases.

\subsection{Non-positively curved simple complexes of finite groups}

In general the question whether a given simple complex of finite groups is (strictly) developable is difficult, and it may be even more difficult to check if the standard development is a model for $\E{G}$. The theory of \emph{non-positively curved} simple complexes of finite groups gives criteria to answer both question in the positive. 
For the definition of non-positively curved simple complexes of finite groups we refer the reader to \cite[II.12]{BH}. The crucial theorem \cite[Theorem~II.12.28]{BH} states that a non-positively curved simple complex of finite groups $G(\mathcal Q)$ is strictly developable and that the standard development $D(\D, G(\mathcal{Q}))$ admits a $\mathrm{CAT}(0)$ metric such that $G$ acts by isometries.
This in particular implies that $D(\D, G(\mathcal{Q}))$ is a model for $\E{G}$ \cite[Corollary~II.2.8(1)]{BH}, and thus Theorem~\ref{thm:main} applies to $G(\mathcal{Q})$. 


\subsection{Automorphism groups of buildings}\label{subsec:autbuildings} A large class of groups that arise as fundamental groups of non-positively curved simple complexes of groups are the automorphism groups of buildings. Before proving Theorem~\ref{thm:building} we need to recall some
terminology. Our exposition loosely follows \cite[Section~5]{DMP}.

A Coxeter system $(W, S)$ is a group $W$ (called a \emph{Coxeter group}) generated by a finite set $S$ and given by the following presentation
\[W=\langle S \mid (s_is_j)^{m_{ij}} \text{ for all } s_i, s_j \in S \rangle\]
where $m_{ii} = 1$ for all $i$, and $m_{ij} = m_{ji} \in \{ 2, 3, \ldots, \infty\}$ (by $m_{ij} = \infty$ we mean that there
is no relation between $s_i$ and $s_j$). We refer the reader to \cite{Davbook} for a detailed treatment of Coxeter groups.

For a subset $J\subset S$ let $W_J$ denote the subgroup of $W$ generated by $J$ (we set $W_{\emptyset}= \{e\}$). If $W_J$ is finite then we call it a \emph{special spherical subgroup} of $W$. Let $\mathcal Q$ be the poset \[\{ J \subseteq S \mid W_J \text{ is a special spherical subgroup of } W\}\]
ordered by inclusion. Note that $\mathcal Q$ contains the empty set  $\emptyset$ as the smallest element.\smallskip

Now suppose that we are given a group $G$ together with subgroups $B$ and $\{P_s\}_{s \in S}$ such that $B \subset P_s$ for every $s \in S$.
For any subset $J \subseteq S$ define the \emph{standard parabolic subgroup} $P_J$ as \[P_J = \langle P_s \mid s \in J \rangle \subseteq G.\]
A coset $gP_J/B$ is called a $J$--\emph{residue}. Assume now that $C=(G,B, \{P_s\}_{s \in S})$ is a building of type $(W,S)$ (see \cite[Example 1.1, \S3]{Davbuild}) and suppose that for every $J \in \mathcal Q$ the subgroup $P_J\subset G$ is finite. In this case $G$ acts properly and chamber transitively on the building $C$. On the other hand, any group acting chamber transitively on a building $C$ of type $(W,S)$ is of the form $(G,B, \{P_s\}_{s \in S})$ where $B$ is the stabiliser of the chamber $c \in C$ and $P_s$ is the stabiliser of the $\{s\}$--residue containing $c$. 
Note that if $(W,S)$ is a Coxeter system then $W$ acts properly and chamber transitively on the building $(W, \{e\}, \{\langle s_i\rangle\}_{s_i \in S})$ of type $(W,S)$. 

We need the following basic lemma.

\begin{lemma}\label{lem:buildingmin} In the above setting, if $J <T$, then $P_{J} \subset P_T$ is a proper inclusion.

\begin{proof}
Identify the set of chambers of $C$ with $G/B$. By definition of a building \cite[\S3]{Davbuild}, there is a $W$--valued distance function \[\delta \colon C \times C \to W\] such that chambers $c,c' \in C$ belong to the same $T$--residue if and only if $\delta(c,c') \in W_T$. Consider the building $W=(W, \{e\}, \{\langle s_i\rangle\}_{s_i \in S})$ and pick a chamber $w \in W$ with $w \in W_T \smallsetminus W_{J}$. Let $\alpha$ be a $W$--isometry $\alpha \colon W \to C$
given by $\alpha(e)=eB$ (see \cite[\S6]{Davbuild}). By definition of a $W$--isometry we have
\[\delta(\alpha(e), \alpha(w))=w \in W_T,\]
and thus $\alpha(e)=eB$ and $\alpha(w)$ both belong to the $T$--residue $P_T/B$. Thus $\alpha(w)=hB$ for some $h \in P_T$. Now suppose that $h \in P_{J}$. In that case, $eB$ and $hB$ belong to the same $J$--residue of $C$ and thus $w=\delta(eB,hB) \in W_{J}$ which contradicts the choice of $w$.
\end{proof}
\end{lemma}

Now consider a poset \begin{equation*}\label{eq:bigposet} \mathcal P = \{gP_J \mid g \in G, J \in \mathcal Q\}\end{equation*} where the partial order is given by the inclusion of cosets.
 Notice that there is an inclusion of posets $i \colon \mathcal Q \hookrightarrow \mathcal P$ given by $J \mapsto P_J$. The geometric realisation $\abs{\mathcal{P}}$ of $\mathcal{P}$ is called the \emph{geometric realisation} of the building $(G,B, \{P_s\}_{s \in S})$. There is a $G$--action on $\abs{\mathcal{P}}$ given by $g \cdot g'P_J = gg'P_J$. One verifies that the stabilisers of this action are the conjugates of groups $P_J$ for $J \in \mathcal Q$, and that the subcomplex $\abs{i(\mathcal Q)} \cong \abs{\mathcal Q}$ is the strict fundamental domain. Now by \cite[Theorem~11.1]{Davbuild} there is a complete $\mathrm{CAT}(0)$ metric on $\abs{P}$ such that $G$ acts by isometries. Thus by the above considerations and \cite[Corollary~12.22]{BH} we obtain that $G$ is the fundamental group of a complex of groups $G(\mathcal Q)$ over $\mathcal Q$ where the local group at $J \in \mathcal Q$ is $P_J$ (the local group at $\emptyset \in Q$ is $B$) and the map $ \phi_{TJ} \colon P_J \to P_T$ for $J <T$ is the inclusion $P_J \subset P_T$. 
 Observe that by Lemma~\ref{lem:buildingmin}, the complex $G(\mathcal Q)$ satisfies the assumptions of Definition~\ref{def:scofg}(\ref{it:scofg2}).

Moreover, if we let $\D$ be the standard panel complex associated to $\mathcal Q$ (see Definition~\ref{def:daviscx}) then one easily verifies that the basic construction $D(\D, G(\mathcal{Q}))$ is naturally homeomorphic to $|\mathcal P|$, see \cite[\S10]{Davbuild}.

\begin{theorem}\label{thm:building}Let $G$ be a group acting properly and chamber transitively on a building of type $(W,S)$, and let $G(\mQ)$ be the associated simple complex of groups. Then $D(B, G(\mathcal{Q}))$ is a model for $\underline{E}G$ satisfying 
\[\mathrm{dim}(D(\B, G(\mathcal{Q}))) = \left\{ \begin{array}{lcr}
      \vcd W        & \text{ if } & \vcd W \neq 2,\\
      2 \text{ or } 3 & \text{ if } & \vcd W =2. \\

      \end{array} \right.\]
\end{theorem}

\begin{proof} By the discussion above we get that $G$ acts properly and cocompactly by isometries on a $\mathrm{CAT}(0)$ space $\abs{\mathcal P}$. Thus by \cite[Corollary~II.2.8(1)]{BH} we get that $\abs{\mathcal P} \cong D(\D, G(\mathcal{Q}))$ is a model for $\ue G$. By \cite[Theorem~5.4]{DMP} we have $\ucd G=\vcd W$. The claim follows from Theorem~\ref{thm:main}.\end{proof}

Note that if $W$ is a Coxeter group then Theorem~\ref{thm:building} implies that $D(\B, W(\mathcal{Q}))$ is a model for $\E{W}$ of dimension equal to $\vcd W$ (except  when $\vcd W=2$). We remark that in this case, the basic construction $D(\D, W(\mathcal{Q}))$ is the so-called \emph{Davis complex} $\Sigma_W$ of $W$ (see \cite[Chapter~7]{Davbook}), and the basic construction $D(\B, W(\mathcal{Q}))$ is the original Bestvina complex $B(W,S)$ constructed in \cite{Best}.\medskip

A large class of Coxeter groups that is well suited for constructing examples is the class of \emph{right-angled} Coxeter groups.

\begin{definition}Let $L$ be a finite flag simplicial complex. The \emph{right-angled Coxeter group} associated to $L$ is the group  $W_L$ generated by involutions corresponding bijectively to the vertices of $L$ subject to the relations that two involutions commute if and only if the corresponding vertices are connected by an edge in $L$.\end{definition}

Note that every special spherical subgroup of $W_L$ is of the form ${C_2}^n$ where the generators correspond to vertices of $L$ that form an $(n-1)$--simplex. Thus the poset $\mathcal Q$ of special spherical subgroups of $W_L$ is isomorphic to the poset of simplices of $L$ with the additional smallest element added, namely the element corresponding to the trivial subgroup. Consequently, the complex $\D = \abs{\mathcal Q}$ is equal to the cone on the barycentric subdivision of $L$. 

Many of the examples presented later in this section are the right-angled Coxeter groups (or some of their variations).

\subsection{Graph products of finite groups} An example of a group acting properly and chamber transitively
 on a building is a \emph{graph product} of finite groups. 
Let $L$ be a finite simplicial graph on the vertex set $S$ and suppose that we are given a finite group $P_s$ for every $s \in S$. Then the graph product $G$ is defined as the quotient of the free product of groups $P_s$ for $s \in S$ by the relations \[\text{ $\{ [P_s,P_t]$ if $[s,t]$ is an edge of $L \}$}.\] In other words, elements of subgroups $P_s$ and $P_t$ commute if and only if there is an edge $[s,t]$ in $L$. The group $G$ acts properly and chamber transitively on a building of type $(W,S)$ where $W$ is a right-angled Coxeter group corresponding to graph $L$ \cite[Theorem~5.1]{Davbuild}.
Note that there is a surjection $G \to \prod_{s\in S}P_s$, and its kernel acts freely on the geometric realisation of the building, which is contractible. This implies that $G$ is virtually torsion-free. Since by \cite[Corollary~5.5]{DMP} we have that $\vcd G=\ucd G$, we conclude that Theorem~\ref{thm:building} gives a model for $\ue G$ of dimension equal to $\vcd G$ (except the case where $\vcd G=2$).



\subsection{Concrete examples of Bestvina complex}
The Bestvina complex can be effectively used for computations of Bredon cohomology (with any coefficient system) of the associated fundamental group. In particular, this complex is better suited for computations than the standard complex $\D$. Not only does $\B$ have smaller dimension than $\D$, but almost always it has a simpler cell structure. 

Below we present a few examples showing these features of $\B$.  All examples are the right-angled Coxeter groups $W_L$ associated to various flag complexes $L$. In each case let $\mathcal Q$ denote the poset of special spherical subgroups and let $W_L(\mathcal{Q})$ denote the associated simple complex of groups. Recall that in this case the complex $\D$ is equal to the cone on the barycentric subdivision of $L$.


\begin{example}Let $L$ be a disjoint union of a vertex and an edge. 
One easily sees that we have $\mathrm{dim}(D(\B, W_L(\mathcal Q))) =\vcd W_L =1$ and $\mathrm{dim}(D(\D, W_L(\mathcal Q)))=2$. In Figure~\ref{fig:disjointedge} we present basic constructions $D(\D, W_L(\mathcal{Q}))$ and $D(\B, W_L(\mathcal{Q}))$.
\end{example}

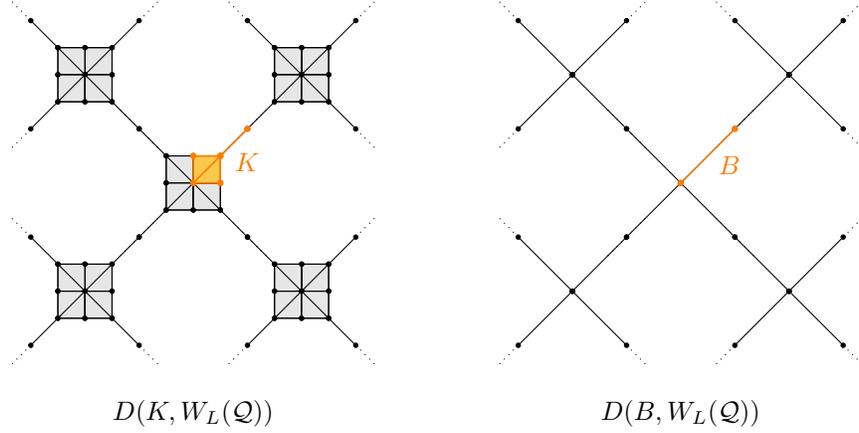
\begin{figure}[!h]
\centering
\begin{tikzpicture}[scale=0.36]

\definecolor{vlgray}{RGB}{230,230,230}
\definecolor{tpurple}{RGB}{250,125,000}
\definecolor{lpurple}{RGB}{250,200,075}

\tikzset{thick/.style={line width=.4pt}}

\tikzset{vth/.style={line width=.6pt}}

\begin{scope}[shift={(-9,0)}]


\draw[fill=vlgray] (0,0) -- (1,0)--(1,1)--(0,1)--(0,0);

\draw[thick] (0,0) -- (1,0)--(1,1)--(0,1)--(0,0);

\draw[thick] (0,0)--(1,1)--(2,2);

\draw[fill] (0,0)   circle [radius=0.08];
\draw[fill] (1,0)   circle [radius=0.08];
\draw[fill] (1,1)   circle [radius=0.08];
\draw[fill] (0,1)   circle [radius=0.08];
\draw[fill] (1,1)   circle [radius=0.08];
\draw[fill] (2,2)   circle [radius=0.08];

\draw[thick, dotted] (2,2)--(2.75,2.75);


\begin{scope}[rotate=90]
\draw[fill=vlgray] (0,0) -- (1,0)--(1,1)--(0,1)--(0,0);

\draw[thick] (0,0) -- (1,0)--(1,1)--(0,1)--(0,0);

\draw[thick] (0,0)--(1,1)--(2,2);

\draw[fill] (0,0)   circle [radius=0.08];
\draw[fill] (1,0)   circle [radius=0.08];
\draw[fill] (1,1)   circle [radius=0.08];
\draw[fill] (0,1)   circle [radius=0.08];
\draw[fill] (1,1)   circle [radius=0.08];
\draw[fill] (2,2)   circle [radius=0.08];

\draw[thick, dotted] (2,2)--(2.75,2.75);

\end{scope}

\begin{scope}[rotate=180]
\draw[fill=vlgray] (0,0) -- (1,0)--(1,1)--(0,1)--(0,0);

\draw[thick] (0,0) -- (1,0)--(1,1)--(0,1)--(0,0);

\draw[thick] (0,0)--(1,1)--(2,2);

\draw[fill] (0,0)   circle [radius=0.08];
\draw[fill] (1,0)   circle [radius=0.08];
\draw[fill] (1,1)   circle [radius=0.08];
\draw[fill] (0,1)   circle [radius=0.08];
\draw[fill] (1,1)   circle [radius=0.08];
\draw[fill] (2,2)   circle [radius=0.08];

\draw[thick, dotted] (2,2)--(2.75,2.75);

\end{scope}

\begin{scope}[rotate=270]
\draw[fill=vlgray] (0,0) -- (1,0)--(1,1)--(0,1)--(0,0);

\draw[thick] (0,0) -- (1,0)--(1,1)--(0,1)--(0,0);

\draw[thick] (0,0)--(1,1)--(2,2);

\draw[fill] (0,0)   circle [radius=0.08];
\draw[fill] (1,0)   circle [radius=0.08];
\draw[fill] (1,1)   circle [radius=0.08];
\draw[fill] (0,1)   circle [radius=0.08];
\draw[fill] (1,1)   circle [radius=0.08];
\draw[fill] (2,2)   circle [radius=0.08];

\draw[thick, dotted] (2,2)--(2.75,2.75);

\end{scope}


\begin{scope}[shift={(-4,-4)}]


\draw[fill=vlgray] (0,0) -- (1,0)--(1,1)--(0,1)--(0,0);

\draw[thick] (0,0) -- (1,0)--(1,1)--(0,1)--(0,0);

\draw[thick] (0,0)--(1,1)--(2,2);

\draw[fill] (0,0)   circle [radius=0.08];
\draw[fill] (1,0)   circle [radius=0.08];
\draw[fill] (1,1)   circle [radius=0.08];
\draw[fill] (0,1)   circle [radius=0.08];
\draw[fill] (1,1)   circle [radius=0.08];
\draw[fill] (2,2)   circle [radius=0.08];

\draw[thick, dotted] (2,2)--(2.75,2.75);


\begin{scope}[rotate=90]
\draw[fill=vlgray] (0,0) -- (1,0)--(1,1)--(0,1)--(0,0);

\draw[thick] (0,0) -- (1,0)--(1,1)--(0,1)--(0,0);

\draw[thick] (0,0)--(1,1)--(2,2);

\draw[fill] (0,0)   circle [radius=0.08];
\draw[fill] (1,0)   circle [radius=0.08];
\draw[fill] (1,1)   circle [radius=0.08];
\draw[fill] (0,1)   circle [radius=0.08];
\draw[fill] (1,1)   circle [radius=0.08];
\draw[fill] (2,2)   circle [radius=0.08];

\draw[thick, dotted] (2,2)--(2.75,2.75);

\end{scope}

\begin{scope}[rotate=180]
\draw[fill=vlgray] (0,0) -- (1,0)--(1,1)--(0,1)--(0,0);

\draw[thick] (0,0) -- (1,0)--(1,1)--(0,1)--(0,0);

\draw[thick] (0,0)--(1,1)--(2,2);

\draw[fill] (0,0)   circle [radius=0.08];
\draw[fill] (1,0)   circle [radius=0.08];
\draw[fill] (1,1)   circle [radius=0.08];
\draw[fill] (0,1)   circle [radius=0.08];
\draw[fill] (1,1)   circle [radius=0.08];
\draw[fill] (2,2)   circle [radius=0.08];

\draw[thick, dotted] (2,2)--(2.75,2.75);

\end{scope}

\begin{scope}[rotate=270]
\draw[fill=vlgray] (0,0) -- (1,0)--(1,1)--(0,1)--(0,0);

\draw[thick] (0,0) -- (1,0)--(1,1)--(0,1)--(0,0);

\draw[thick] (0,0)--(1,1)--(2,2);

\draw[fill] (0,0)   circle [radius=0.08];
\draw[fill] (1,0)   circle [radius=0.08];
\draw[fill] (1,1)   circle [radius=0.08];
\draw[fill] (0,1)   circle [radius=0.08];
\draw[fill] (1,1)   circle [radius=0.08];
\draw[fill] (2,2)   circle [radius=0.08];

\draw[thick, dotted] (2,2)--(2.75,2.75);

\end{scope}


\end{scope}

\begin{scope}[shift={(4,4)}]


\draw[fill=vlgray] (0,0) -- (1,0)--(1,1)--(0,1)--(0,0);

\draw[thick] (0,0) -- (1,0)--(1,1)--(0,1)--(0,0);

\draw[thick] (0,0)--(1,1)--(2,2);

\draw[fill] (0,0)   circle [radius=0.08];
\draw[fill] (1,0)   circle [radius=0.08];
\draw[fill] (1,1)   circle [radius=0.08];
\draw[fill] (0,1)   circle [radius=0.08];
\draw[fill] (1,1)   circle [radius=0.08];
\draw[fill] (2,2)   circle [radius=0.08];

\draw[thick, dotted] (2,2)--(2.75,2.75);


\begin{scope}[rotate=90]
\draw[fill=vlgray] (0,0) -- (1,0)--(1,1)--(0,1)--(0,0);

\draw[thick] (0,0) -- (1,0)--(1,1)--(0,1)--(0,0);

\draw[thick] (0,0)--(1,1)--(2,2);

\draw[fill] (0,0)   circle [radius=0.08];
\draw[fill] (1,0)   circle [radius=0.08];
\draw[fill] (1,1)   circle [radius=0.08];
\draw[fill] (0,1)   circle [radius=0.08];
\draw[fill] (1,1)   circle [radius=0.08];
\draw[fill] (2,2)   circle [radius=0.08];

\draw[thick, dotted] (2,2)--(2.75,2.75);

\end{scope}

\begin{scope}[rotate=180]
\draw[fill=vlgray] (0,0) -- (1,0)--(1,1)--(0,1)--(0,0);

\draw[thick] (0,0) -- (1,0)--(1,1)--(0,1)--(0,0);

\draw[thick] (0,0)--(1,1)--(2,2);

\draw[fill] (0,0)   circle [radius=0.08];
\draw[fill] (1,0)   circle [radius=0.08];
\draw[fill] (1,1)   circle [radius=0.08];
\draw[fill] (0,1)   circle [radius=0.08];
\draw[fill] (1,1)   circle [radius=0.08];
\draw[fill] (2,2)   circle [radius=0.08];

\draw[thick, dotted] (2,2)--(2.75,2.75);

\end{scope}

\begin{scope}[rotate=270]
\draw[fill=vlgray] (0,0) -- (1,0)--(1,1)--(0,1)--(0,0);

\draw[thick] (0,0) -- (1,0)--(1,1)--(0,1)--(0,0);

\draw[thick] (0,0)--(1,1)--(2,2);

\draw[fill] (0,0)   circle [radius=0.08];
\draw[fill] (1,0)   circle [radius=0.08];
\draw[fill] (1,1)   circle [radius=0.08];
\draw[fill] (0,1)   circle [radius=0.08];
\draw[fill] (1,1)   circle [radius=0.08];
\draw[fill] (2,2)   circle [radius=0.08];

\draw[thick, dotted] (2,2)--(2.75,2.75);

\end{scope}


\end{scope}

\begin{scope}[shift={(4,-4)}]


\draw[fill=vlgray] (0,0) -- (1,0)--(1,1)--(0,1)--(0,0);

\draw[thick] (0,0) -- (1,0)--(1,1)--(0,1)--(0,0);

\draw[thick] (0,0)--(1,1)--(2,2);

\draw[fill] (0,0)   circle [radius=0.08];
\draw[fill] (1,0)   circle [radius=0.08];
\draw[fill] (1,1)   circle [radius=0.08];
\draw[fill] (0,1)   circle [radius=0.08];
\draw[fill] (1,1)   circle [radius=0.08];
\draw[fill] (2,2)   circle [radius=0.08];

\draw[thick, dotted] (2,2)--(2.75,2.75);


\begin{scope}[rotate=90]
\draw[fill=vlgray] (0,0) -- (1,0)--(1,1)--(0,1)--(0,0);

\draw[thick] (0,0) -- (1,0)--(1,1)--(0,1)--(0,0);

\draw[thick] (0,0)--(1,1)--(2,2);

\draw[fill] (0,0)   circle [radius=0.08];
\draw[fill] (1,0)   circle [radius=0.08];
\draw[fill] (1,1)   circle [radius=0.08];
\draw[fill] (0,1)   circle [radius=0.08];
\draw[fill] (1,1)   circle [radius=0.08];
\draw[fill] (2,2)   circle [radius=0.08];

\draw[thick, dotted] (2,2)--(2.75,2.75);

\end{scope}

\begin{scope}[rotate=180]
\draw[fill=vlgray] (0,0) -- (1,0)--(1,1)--(0,1)--(0,0);

\draw[thick] (0,0) -- (1,0)--(1,1)--(0,1)--(0,0);

\draw[thick] (0,0)--(1,1)--(2,2);

\draw[fill] (0,0)   circle [radius=0.08];
\draw[fill] (1,0)   circle [radius=0.08];
\draw[fill] (1,1)   circle [radius=0.08];
\draw[fill] (0,1)   circle [radius=0.08];
\draw[fill] (1,1)   circle [radius=0.08];
\draw[fill] (2,2)   circle [radius=0.08];

\draw[thick, dotted] (2,2)--(2.75,2.75);

\end{scope}

\begin{scope}[rotate=270]
\draw[fill=vlgray] (0,0) -- (1,0)--(1,1)--(0,1)--(0,0);

\draw[thick] (0,0) -- (1,0)--(1,1)--(0,1)--(0,0);

\draw[thick] (0,0)--(1,1)--(2,2);

\draw[fill] (0,0)   circle [radius=0.08];
\draw[fill] (1,0)   circle [radius=0.08];
\draw[fill] (1,1)   circle [radius=0.08];
\draw[fill] (0,1)   circle [radius=0.08];
\draw[fill] (1,1)   circle [radius=0.08];
\draw[fill] (2,2)   circle [radius=0.08];

\draw[thick, dotted] (2,2)--(2.75,2.75);

\end{scope}


\end{scope}

\begin{scope}[shift={(-4,4)}]


\draw[fill=vlgray] (0,0) -- (1,0)--(1,1)--(0,1)--(0,0);

\draw[thick] (0,0) -- (1,0)--(1,1)--(0,1)--(0,0);

\draw[thick] (0,0)--(1,1)--(2,2);

\draw[fill] (0,0)   circle [radius=0.08];
\draw[fill] (1,0)   circle [radius=0.08];
\draw[fill] (1,1)   circle [radius=0.08];
\draw[fill] (0,1)   circle [radius=0.08];
\draw[fill] (1,1)   circle [radius=0.08];
\draw[fill] (2,2)   circle [radius=0.08];

\draw[thick, dotted] (2,2)--(2.75,2.75);


\begin{scope}[rotate=90]
\draw[fill=vlgray] (0,0) -- (1,0)--(1,1)--(0,1)--(0,0);

\draw[thick] (0,0) -- (1,0)--(1,1)--(0,1)--(0,0);

\draw[thick] (0,0)--(1,1)--(2,2);

\draw[fill] (0,0)   circle [radius=0.08];
\draw[fill] (1,0)   circle [radius=0.08];
\draw[fill] (1,1)   circle [radius=0.08];
\draw[fill] (0,1)   circle [radius=0.08];
\draw[fill] (1,1)   circle [radius=0.08];
\draw[fill] (2,2)   circle [radius=0.08];

\draw[thick, dotted] (2,2)--(2.75,2.75);

\end{scope}

\begin{scope}[rotate=180]
\draw[fill=vlgray] (0,0) -- (1,0)--(1,1)--(0,1)--(0,0);

\draw[thick] (0,0) -- (1,0)--(1,1)--(0,1)--(0,0);

\draw[thick] (0,0)--(1,1)--(2,2);

\draw[fill] (0,0)   circle [radius=0.08];
\draw[fill] (1,0)   circle [radius=0.08];
\draw[fill] (1,1)   circle [radius=0.08];
\draw[fill] (0,1)   circle [radius=0.08];
\draw[fill] (1,1)   circle [radius=0.08];
\draw[fill] (2,2)   circle [radius=0.08];

\draw[thick, dotted] (2,2)--(2.75,2.75);

\end{scope}

\begin{scope}[rotate=270]
\draw[fill=vlgray] (0,0) -- (1,0)--(1,1)--(0,1)--(0,0);

\draw[thick] (0,0) -- (1,0)--(1,1)--(0,1)--(0,0);

\draw[thick] (0,0)--(1,1)--(2,2);

\draw[fill] (0,0)   circle [radius=0.08];
\draw[fill] (1,0)   circle [radius=0.08];
\draw[fill] (1,1)   circle [radius=0.08];
\draw[fill] (0,1)   circle [radius=0.08];
\draw[fill] (1,1)   circle [radius=0.08];
\draw[fill] (2,2)   circle [radius=0.08];

\draw[thick, dotted] (2,2)--(2.75,2.75);

\end{scope}


\end{scope}

\draw[fill=lpurple] (0,0) -- (1,0)--(1,1)--(0,1)--(0,0);

\draw[vth, tpurple] (0,0) -- (1,0)--(1,1)--(0,1)--(0,0);

\draw[vth, tpurple] (0,0)--(1,1)--(2,2);

\draw[fill,tpurple] (0,0)   circle [radius=0.1];
\draw[fill,tpurple] (1,0)   circle [radius=0.1];
\draw[fill,tpurple] (1,1)   circle [radius=0.1];
\draw[fill,tpurple] (0,1)   circle [radius=0.1];
\draw[fill,tpurple] (1,1)   circle [radius=0.1];
\draw[fill,tpurple] (2,2)   circle [radius=0.1];



\node    at (0,-8.5)  {$D(\D, W_L(\mathcal Q))$};

\node [tpurple]    at (2,0.75)  {$\D$};

\end{scope}

\begin{scope}[shift={(9,0)}]

\begin{scope}

\draw[thick] (0,0)--(2,2);

\draw[fill] (0,0)   circle [radius=0.08];
\draw[fill] (2,2)   circle [radius=0.08];

\draw[thick, dotted] (2,2)--(2.75,2.75);

\end{scope}

\begin{scope}[rotate=90]

\draw[thick] (0,0)--(2,2);

\draw[fill] (0,0)   circle [radius=0.08];
\draw[fill] (2,2)   circle [radius=0.08];

\draw[thick, dotted] (2,2)--(2.75,2.75);

\end{scope}

\begin{scope}[rotate=180]

\draw[thick] (0,0)--(2,2);

\draw[fill] (0,0)   circle [radius=0.08];
\draw[fill] (2,2)   circle [radius=0.08];

\draw[thick, dotted] (2,2)--(2.75,2.75);

\end{scope}

\begin{scope}[rotate=270]

\draw[thick] (0,0)--(2,2);

\draw[fill] (0,0)   circle [radius=0.08];
\draw[fill] (2,2)   circle [radius=0.08];

\draw[thick, dotted] (2,2)--(2.75,2.75);
\end{scope}


\begin{scope}[shift={(4,4)}]

\begin{scope}

\draw[thick] (0,0)--(2,2);

\draw[fill] (0,0)   circle [radius=0.08];
\draw[fill] (2,2)   circle [radius=0.08];

\draw[thick, dotted] (2,2)--(2.75,2.75);

\end{scope}

\begin{scope}[rotate=90]

\draw[thick] (0,0)--(2,2);

\draw[fill] (0,0)   circle [radius=0.08];
\draw[fill] (2,2)   circle [radius=0.08];

\draw[thick, dotted] (2,2)--(2.75,2.75);
\end{scope}

\begin{scope}[rotate=180]

\draw[thick] (0,0)--(2,2);

\draw[fill] (0,0)   circle [radius=0.08];
\draw[fill] (2,2)   circle [radius=0.08];

\draw[thick, dotted] (2,2)--(2.75,2.75);

\end{scope}

\begin{scope}[rotate=270]

\draw[thick] (0,0)--(2,2);

\draw[fill] (0,0)   circle [radius=0.08];
\draw[fill] (2,2)   circle [radius=0.08];

\draw[thick, dotted] (2,2)--(2.75,2.75);
\end{scope}


\end{scope}

\begin{scope}[shift={(-4,-4)}]

\begin{scope}

\draw[thick] (0,0)--(2,2);

\draw[fill] (0,0)   circle [radius=0.08];
\draw[fill] (2,2)   circle [radius=0.08];

\draw[thick, dotted] (2,2)--(2.75,2.75);
\end{scope}

\begin{scope}[rotate=90]

\draw[thick] (0,0)--(2,2);

\draw[fill] (0,0)   circle [radius=0.08];
\draw[fill] (2,2)   circle [radius=0.08];

\draw[thick, dotted] (2,2)--(2.75,2.75);
\end{scope}

\begin{scope}[rotate=180]

\draw[thick] (0,0)--(2,2);

\draw[fill] (0,0)   circle [radius=0.08];
\draw[fill] (2,2)   circle [radius=0.08];

\draw[thick, dotted] (2,2)--(2.75,2.75);

\end{scope}

\begin{scope}[rotate=270]

\draw[thick] (0,0)--(2,2);

\draw[fill] (0,0)   circle [radius=0.08];
\draw[fill] (2,2)   circle [radius=0.08];

\draw[thick, dotted] (2,2)--(2.75,2.75);
\end{scope}


\end{scope}

\begin{scope}[shift={(4,-4)}]

\begin{scope}

\draw[thick] (0,0)--(2,2);

\draw[fill] (0,0)   circle [radius=0.08];
\draw[fill] (2,2)   circle [radius=0.08];

\draw[thick, dotted] (2,2)--(2.75,2.75);
\end{scope}

\begin{scope}[rotate=90]

\draw[thick] (0,0)--(2,2);

\draw[fill] (0,0)   circle [radius=0.08];
\draw[fill] (2,2)   circle [radius=0.08];

\draw[thick, dotted] (2,2)--(2.75,2.75);
\end{scope}

\begin{scope}[rotate=180]

\draw[thick] (0,0)--(2,2);

\draw[fill] (0,0)   circle [radius=0.08];
\draw[fill] (2,2)   circle [radius=0.08];

\draw[thick, dotted] (2,2)--(2.75,2.75);

\end{scope}

\begin{scope}[rotate=270]

\draw[thick] (0,0)--(2,2);

\draw[fill] (0,0)   circle [radius=0.08];
\draw[fill] (2,2)   circle [radius=0.08];

\draw[thick, dotted] (2,2)--(2.75,2.75);
\end{scope}


\end{scope}

\begin{scope}[shift={(-4,4)}]

\begin{scope}

\draw[thick] (0,0)--(2,2);

\draw[fill] (0,0)   circle [radius=0.08];
\draw[fill] (2,2)   circle [radius=0.08];

\draw[thick, dotted] (2,2)--(2.75,2.75);
\end{scope}

\begin{scope}[rotate=90]

\draw[thick] (0,0)--(2,2);

\draw[fill] (0,0)   circle [radius=0.08];
\draw[fill] (2,2)   circle [radius=0.08];

\draw[thick, dotted] (2,2)--(2.75,2.75);
\end{scope}

\begin{scope}[rotate=180]

\draw[thick] (0,0)--(2,2);

\draw[fill] (0,0)   circle [radius=0.08];
\draw[fill] (2,2)   circle [radius=0.08];

\draw[thick, dotted] (2,2)--(2.75,2.75);

\end{scope}

\begin{scope}[rotate=270]

\draw[thick] (0,0)--(2,2);

\draw[fill] (0,0)   circle [radius=0.08];
\draw[fill] (2,2)   circle [radius=0.08];

\draw[thick, dotted] (2,2)--(2.75,2.75);
\end{scope}


\end{scope}

\draw[vth, tpurple] (0,0)--(2,2);

\draw[fill, tpurple] (0,0)   circle [radius=0.1];
\draw[fill,tpurple] (2,2)   circle [radius=0.1];

\node    at (0,-8.5)  {$D(\B, W_L(\mathcal Q))$};

\node [tpurple]    at (1.8,0.65)  {$\B$};

\end{scope}

\end{tikzpicture}
\caption{Basic constructions $D(\D, W_L(\mathcal Q))$ and $D(\B, W_L(\mathcal Q))$. Panel complexes $\D$ and $\B$ are orange.}
\label{fig:disjointedge}
\end{figure}

\begin{example}The following example shows that besides lower dimension, the cell structure of $\B$ is significantly simpler than the cell structure of $\D$. Let $L$ be a flag complex which is a hexagon built out of six triangles. In this case we have $\mathrm{dim}(D(\D, W_L(\mathcal Q)))=3$ and $\mathrm{dim}(D(\B, W_L(\mathcal Q)))=\vcd W_L=2$. Panel complexes $\D$ and $\B$ are shown in Figure~\ref{fig:simplecell}.
\end{example}

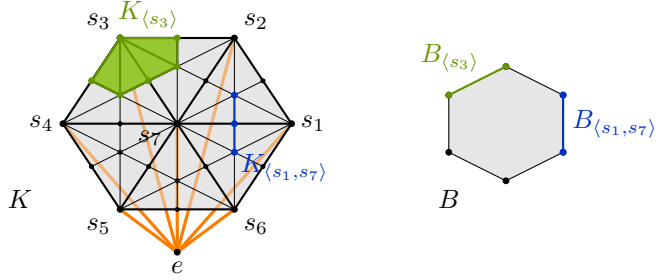
\begin{figure}[!h]
\centering

\tikzset{
  foo/.style={very thin, color=#1},
  sru/.style={very thick, color=#1},
}

\begin{tikzpicture}[scale=0.38]
\definecolor{torange}{RGB}{250,125,000}
\definecolor{lorange}{RGB}{250,200,075}

\definecolor{vlgray}{RGB}{230,230,230}
\definecolor{tpurple}{RGB}{100,150,00}
\definecolor{lpurple}{RGB}{150,200,050}

\definecolor{red}{RGB}{000,50,200}

\begin{scope}[shift={(3.5,0)}]
\draw[fill, vlgray] (2,1)--(0,2)--(-2,1)--(-2,-1)--(0,-2)--(2,-1)--(2,1);
\draw (2,1)--(0,2)--(-2,1)--(-2,-1)--(0,-2)--(2,-1)--(2,1);

\draw[fill]  (2,1)  circle [radius=0.1];
\draw[fill]  (0,2)  circle [radius=0.1];
\draw[fill]  (-2,1)  circle [radius=0.1];
\draw[fill]  (-2,-1)  circle [radius=0.1];
\draw[fill]  (0,-2)  circle [radius=0.1];
\draw[fill]  (2,-1)  circle [radius=0.1];


 \node [black, below ] at (-2,-2)   {$\B$};

\draw[thick, tpurple] (-2,1)--(0,2);

\draw[fill, tpurple]  (-2,1) circle [radius=0.1];
\draw[fill, tpurple]  (0,2) circle [radius=0.1];

\node [tpurple, above left ] at (-0.6,1.5)   {$\B_{\langle s_3\rangle}$};

\draw[thick, red] (2,-1)--(2,0)--(2,1);
\draw[fill, red]  (2,-1) circle [radius=0.1];
\draw[fill, red]  (2,1) circle [radius=0.1];

\node [red, right] at (2,0)   {$\B_{\langle s_1, s_7 \rangle}$};

\end{scope}

\begin{scope}[shift={(-8,0)}]

 \node [black, below ] at (-5.5,-2)   {$\D$};


\draw[fill, vlgray] (-4,0)--(-2,3)--(2,3)--(4,0)--(2,-3)--(-2,-3)--(-4,0);



\draw[sru=torange] (-4,0)--(0,-4.5);
\draw[sru=torange](-2,3)  --(0,-4.5);
\draw[sru=torange](-2,-3)--(0,-4.5) ;
\draw[sru=torange](0,0)--(0,-4.5);
\draw[sru=torange](2,3)--(0,-4.5) ; 
\draw[sru=torange](2,-3)--(0,-4.5);
\draw[sru=torange](4,0) --(0,-4.5);

\draw[fill, vlgray, opacity=0.5] (-4,0)--(-2,3)--(2,3)--(4,0)--(2,-3)--(-2,-3)--(-4,0);

  \draw[fill, black]  (0,-4.5)  circle [radius=0.1];
  \node [black, below ] at (0,-4.5)   {$e$};


\draw[thick] (-4,0)--(0,0)--(4,0);
\draw[thick] (-2,3)--(0,0)--(2,-3);
\draw[thick] (2,3)--(0,0)--(-2,-3);
\draw[thick] (-4,0)--(-2,3)--(2,3)--(4,0)--(2,-3)--(-2,-3)--(-4,0);

\draw[fill]  (0,0)  circle [radius=0.1];
\draw[fill]  (-4,0)  circle [radius=0.1];
\draw[fill]  (-2,3)  circle [radius=0.1];
\draw[fill]  (2,3)  circle [radius=0.1];
\draw[fill]  (4,0)  circle [radius=0.1];
\draw[fill]  (2,-3)  circle [radius=0.1];
\draw[fill]  (-2,3)  circle [radius=0.1];
\draw[fill]  (-2,-3)  circle [radius=0.1];

 \draw[fill]  (-2,1)  circle [radius=0.07];
 \draw[fill]  (-2,-1)  circle [radius=0.07];
 \draw[fill]  (2,1)  circle [radius=0.07];
\draw[fill]  (2,-1)  circle [radius=0.07];
 \draw[fill]  (0,2)  circle [radius=0.07];
 \draw[fill]  (0,-2)  circle [radius=0.07];
\draw[fill]  (-3,-1.5)  circle [radius=0.07];
\draw[fill]  (-3,1.5)  circle [radius=0.07];
\draw[fill]  (3,-1.5)  circle [radius=0.07];
\draw[fill]  (3,1.5)  circle [radius=0.07];

 \draw[fill]  (-2,0)  circle [radius=0.07];

  \draw[fill]  (-1,-1.5)  circle [radius=0.07];
\draw[fill]  (1,-1.5)  circle [radius=0.07];
\draw[fill]  (1,1.5)  circle [radius=0.07];
\draw[fill]  (0,-3)  circle [radius=0.07];

\node [below left, black] at (-0.2,0.1)  {$s_7$};

\node [left, black] at (-4,0)  {$s_4$};
\node [black, above left] at (-2,3)   { $s_3$ };
\node [black, above right] at (2,3) {$s_2$};
\node [black, right] at (4,0)  {$s_1$};
\node [black, below right] at (2,-3)  {$s_6$};
\node [black, below left] at (-2,-3)   {$s_5$};



 \draw[very thin] (-3,1.5)--(0,0)--(3,-1.5);
 \draw[very thin] (3,1.5)--(0,0)--(-3,-1.5);
 \draw[very thin] (0,3)--(0,0)--(0,-3);
 \draw[very thin] (2,3)--(2,0)--(2,-3);
 \draw[very thin] (-2,3)--(-2,0)--(-2,-3);

\draw[very thin] (-4,0)--(2,3);
\draw[very thin] (-2,3)--(4,0);

\draw[very thin] (-4,0)--(2,-3);
\draw[very thin] (-2,-3)--(4,0);


\draw[thick, fill=lpurple] (-3,1.5)--(-2,1)--(-1,1.5)--(0,2)--(0,3)--(-2,3)--(-3,1.5);

\draw[thick, tpurple] (-3,1.5)--(-2,1)--(-1,1.5)--(0,2)--(0,3)--(-2,3)--(-3,1.5);
\draw[thick, tpurple] (-2,3)--(-2,1);
\draw[thick, tpurple] (-2,3)--(-1,1.5);
\draw[thick, tpurple] (-2,3)--(0,2);
\draw[fill, tpurple]  (-3,1.5)  circle [radius=0.1];
\draw[fill, tpurple]  (-2,1)  circle [radius=0.1];
\draw[fill, tpurple]  (-1,1.5) circle [radius=0.1];
\draw[fill, tpurple]  (0,2) circle [radius=0.1];
\draw[fill, tpurple]  (0,3) circle [radius=0.1];

\draw[fill, tpurple] (-2,3)  circle [radius=0.1];

\draw[thick, red] (2,-1)--(2,0)--(2,1);
\draw[fill, red]  (2,-1) circle [radius=0.1];
\draw[fill, red]  (2,1) circle [radius=0.1];
\draw[fill, red]  (2,0) circle [radius=0.1];

\node [tpurple, above  ] at (-1,3)   {$\D_{\langle s_3\rangle}$};
\node [red, right] at (1.9,-1.55)   {$\D_{\langle s_1, s_7 \rangle}$};

\end{scope}
\end{tikzpicture}
\caption{Complexes $\D$ and $\B$, and their cell structures. Complex $\B$ has one $2$--cell, six $1$--cells and six $0$--cells. 
Complex $\D$ has thirty-six $3$--cells, ninety-six $2$--cells, eighty-five $1$--cells and twenty-six $0$--cells. Panels of $\D$ and $\B$ corresponding to subgroup $\langle s_3 \rangle $ are green. Panels corresponding to subgroup $\langle s_1,s_7 \rangle $ are blue.}
\label{fig:simplecell}
\end{figure} 


 \begin{example}Finally, note that the discrepancy between dimensions of $\D$ and $\B$ can be arbitrarily large. The extreme example is when $L$ is an $n$--simplex for $n>0$. Then $W_L \cong {C_2}^{n+1}$ and so $\mathrm{dim}(\B)=0$, but $\mathrm{dim}(\D)=n+1$. In order to get an infinite group $W_L$, let $L$ consist of two $n$--simplices sharing a common $k$--simplex for some $0<k<n$. Then \[W_L\cong {C_2}^{n+1} \ast_{{C_2}^{k+1}} {C_2}^{n+1}\] is a virtually free group, so $\mathrm{dim}(\B)=1$ but still $\mathrm{dim}(\D)=n+1$.\end{example}

\subsection{Groups with $\mathrm{vcd}G < \ucd G$ and $\ucd G < \ugd G$.} 

Here we show that various known counterexamples to 
Brown's question and the generalised Eilenberg-Ganea conjecture (see e.g.,\ \cite{BLN}, \cite{LePe}), are the fundamental groups of simple complexes of groups. In particular, these examples show that the dimension bounds in Theorem~\ref{thm:main} are sharp.

\subsubsection*{Groups with $\ucd G < \ugd G$.} Suppose that $L$ is a $2$--dimensional acyclic polyhedron with a finite, non-trivial fundamental group. Let $W_L$ denote the right-angled Coxeter group associated to a flag triangulation of $L$, and let $W_L(\mathcal Q)$ denote the corresponding simple complex of groups. In this case we have $K \cong C(L')$ and thus $\mathrm{dim}(\D)=3$. By \cite[Proposition~5]{BLN} the polyhedron $L$ does not embed into a $2$--dimensional contractible polyhedron and therefore we have  $\mathrm{dim}(\B)=3$ as well (cf.\ Remark~\ref{rem:2dimacyclic}). In fact, by \cite[Proposition~4]{BLN} we have $\ucd W_L=2$ and $\ugd W_L=3$. Combining this with \cite[Theorem~5.4]{DMP}, which says that $\vcd W =\ucd W$ for any Coxeter group, we obtain the following. 

\begin{proposition}\label{prop:BLN}For $W_L$ as above we have
\[2=\vcd W_L =\ucd W_L < \ugd W_L = \mathrm{dim}(\B) =\mathrm{dim}(\D)=3.\]
\end{proposition}

The proposition in particular shows that the dimension bounds in Theorems~\ref{thm:main} and \ref{thm:building} are sharp. An example of a polyhedron $L$ satisfying the assumptions of the proposition is given in Example~\ref{ex:a5action}.

In \cite[Proposition~4]{BLN} the main reason for which $\ugd W_L=3$ is that $L$ does not embed into a $2$--dimensional contractible polyhedron. Since this is the same reason for which $\mathrm{dim}(\B)=3$, the following question seems natural.

\begin{ques}Does there exist a strictly developable simple complex of finite groups $G(\mathcal Q)$ with the fundamental group $G$ such that $\ugd G=2$ but $\mathrm{dim}(D(\B, G(\mathcal Q)))=3$?
\end{ques}

Finally, in \cite[Remark~2]{Best} Bestvina asked whether a Coxeter group $W$ with $\vcd W=2$, admits a $2$--dimensional complex $B$. We remark that the groups $W_L$ described above answer this question in the negative.

\subsubsection*{Groups with $\vcd G < \ucd G$.}

Here we present an example of a group for which $\vcd G=2$ and $\ucd G=3$. This example is a certain finite extension of a right-angled Coxeter group, and originally is due to Leary and the first author \cite[Example~1]{LePe}.\medskip

Let $L$ be a finite flag simplicial complex and let $F$ be a finite group acting on $L$ by simplicial automorphisms. This gives an action of $F$ on the right-angled Coxeter group $W_L$, and thus a semi-direct product $G=W_L \rtimes F$. The group $G$ acts properly on Davis complex $\Sigma_{W_L}$. If the action of $F$ on $L$ has a strict fundamental domain $Y$, then the action of $G$ on $\Sigma_{W_L}$ has a strict fundamental domain which is homeomorphic to the cone $C(Y')$. This implies that $G$ is a fundamental group of a simple complex of groups over a poset corresponding to stabilisers of various subcomplexes of the fundamental domain $C(Y)$ (see \cite[Corollary~II.12.22]{BH}). In favourable cases one can describe this poset explicitly and use it to conclude information about $\vcd G$ and $\ucd G$.

\begin{example}\label{ex:a5action}
We will now outline a construction of an action of the alternating group $F=A_5$ on a flag $2$--complex $L$ which is a triangulation of the $2$--skeleton of the Poincar\'{e} homology sphere (see for example \cite[Example~1]{LePe}) in order to illustrate the underlying simple complex of finite groups $G(\mathcal{Q})$.

For the $1$--skeleton $L^{(1)}$ take the barycentric subdivision of the complete graph on five vertices $v_1, \dots, v_5$ with the permutation action of $A_5$.  The group $A_5$ has twenty-four elements of order $5$. They split into two conjugacy classes of size $12$. Every element of order $5$ is conjugate to its inverse. We fix one of these conjugacy classes and note that this gives six inverse pairs of $5$--cycles $\{\sigma_1, \sigma^{-1}_1\}, \dots, \{\sigma_6, \sigma^{-1}_6\}$. Define $L$, by attaching six $2$--cells $p_1, \dots ,p_6$ using the $5$--cycles $\sigma_1, \dots, \sigma_6$ to describe the attaching maps. Each $2$--cell $p_i$ is a cone on its subdivided pentagonal boundary where $\sigma_i$ acts by fixing the cone point.  The $2$--simplices of  $L$ are sixty right-angled triangles on which $A_5$ acts simply transitively.  

The fundamental domain for the action of $A_5$ on $L$ is a single right-angled triangle $Y$. The fundamental domain for the action of $G=W_L \rtimes A_5$ on $\Sigma_{W_L}$ is homeomorphic to $C(Y')$, the cone on the barycentric subdivision of $Y$. These domains, together with the stabilisers of vertices are presented in Figure~\ref{fig:reflike}.

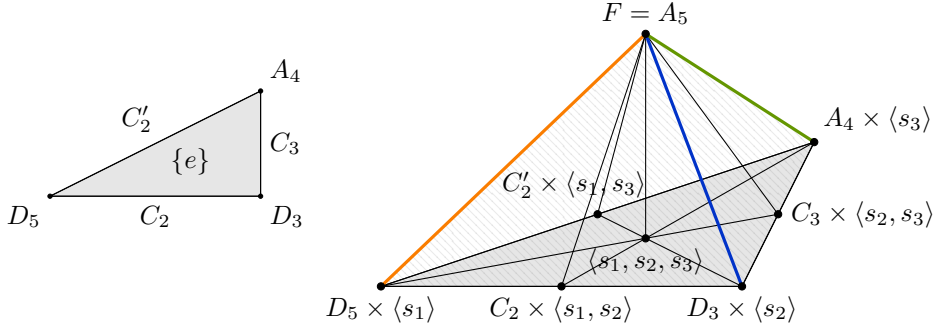
\begin{figure}[!h]
\centering
\begin{tikzpicture}[scale=0.4]

\definecolor{torange}{RGB}{250,125,000}
\definecolor{lorange}{RGB}{250,200,075}

\definecolor{vlgray}{RGB}{230,230,230}
\definecolor{tpurple}{RGB}{100,150,00}
\definecolor{lpurple}{RGB}{150,200,050}

\definecolor{blue}{RGB}{000,50,200}



\begin{scope}[shift={(-7,3)},scale=0.7]

\draw[fill=vlgray] (-5,0)--(5,0)--(5,5)--(-5,0);

\node   at (0.333*5,0.333*5)   {$\{e\}$};

\node [below left]  at (-5,0)   {$D_5$};

\node [below ]  at (0,0)   {$C_2$};

\node [below right]  at (5,0)   {$D_3$};

\node [right]  at (5,2.5)   {$C_3$};
\node [above right]  at (5,5)   {$A_4$};

\node [above left]  at (0.3,2.6)   {$C_2'$};

\draw (-5,0)--(5,0)--(5,5)--(-5,0);

\draw[fill]  (-5,0)  circle [radius=0.1];
\draw[fill]  (5,0)  circle [radius=0.1];
\draw[fill]  (5,5)  circle [radius=0.1];

\end{scope}

\begin{scope}[shift={(6.5,0)},scale=1.2]

\draw[fill=vlgray] (-5,0)--(5,0)--(7,4)--(-5,0);

\draw[pattern=north west lines, pattern color=gray, opacity=0.4] (-5,0)--(5,0)--(7,4)--(0.333*7,7)--(-5,0);


\node [above]  at (0.333*7,7)   {$F=A_5$};

\node [below ]  at (-5,0)   {$D_5 \times \langle s_1 \rangle$};

\node [below ]  at (0,0)   {$C_2 \times \langle s_1,s_2 \rangle$};

\node [below]  at (5,0)   {$D_3 \times \langle s_2 \rangle$};

\node [right]  at (6.1,2)   {$C_3 \times \langle s_2,s_3 \rangle$};

\node [above right]  at (7,4)   {$A_4 \times \langle s_3 \rangle$};

\node [below]  at (0.333*7,0.333*4)   {$\langle s_1,s_2, s_3 \rangle$};

\draw (0.333*7, 7)--(0.333*7,0.333*4);

\draw (-5,0)--(5,0)--(7,4)--(-5,0);

\draw (0,0)--(7,4);
\draw (5,0)--(1,2);
\draw (-5,0)--(6,2);

\draw (0.333*7, 7)--(0,0);
\draw (0.333*7, 7)--(7,4);
\draw (0.333*7, 7)--(5,0);
\draw (0.333*7, 7)--(1,2);
\draw (0.333*7, 7)--(-5,0);
\draw (0.333*7, 7)--(6,2);

\draw[very thick, torange] (-5,0)--(0.333*7,7);
\draw[very thick, blue] (5,0)--(0.333*7,7);
\draw[very thick, tpurple] (7,4)--(0.333*7,7);

\draw[fill]  (0,0)  circle [radius=0.1];

\draw[fill]  (1,2)  circle [radius=0.1];
\draw[fill]  (6,2)  circle [radius=0.1];
\draw[fill]  (0.333*7, 0.333*4)  circle [radius=0.1];

\draw[fill]  (-5,0)  circle [radius=0.1];
\draw[fill]  (5,0)  circle [radius=0.1];
\draw[fill]  (7,4)  circle [radius=0.1];

\draw[fill]  (0.333*7, 7)  circle [radius=0.1];


  \node [above left]  at (2.6,2.2)   {$C_2' \times \langle s_1,s_3 \rangle$};

\end{scope}

\end{tikzpicture}
\caption{Fundamental domains $Y$ (left) and $C(Y')$ (right) together with stabilisers of cells and vertices respectively.}
\label{fig:reflike}
\end{figure}

By \cite[Corollary~II.12.22]{BH} we get that $G$ is the fundamental group of a simple complex of groups $G(\mathcal Q')$, where $\mathcal Q'$ is the poset of simplices of $C(Y')$ and the local group at a simplex $\sigma$ is the intersection of stabilisers of vertices of $\sigma$. In this case the basic construction $D(\D, G(\mathcal Q'))$ (where $\D$ is the standard panel complex associated to $\mathcal Q'$)  is homeomorphic to Davis complex $\Sigma_{W_L}$.
However, the complex $G(\mathcal Q')$ does not satisfy the non-surjectivity assumption of Definition~\ref{def:scofg}.(\ref{it:scofg2}). To remedy this, one defines a new complex of groups $G(\mathcal Q)$ over a new poset $\mathcal Q$, roughly speaking, by identifying elements of $\mathcal Q'$ that have the same local group. The simple complex of groups $G(\mathcal Q)$ has the same fundamental group $G$, the standard complex $\D = \abs{\mathcal Q}$ is homeomorphic to $C(Y')$ and the basic construction $D(\D, G(\mathcal Q))$ is equivariantly homeomorphic to the Davis complex $\Sigma_{W_L}$.

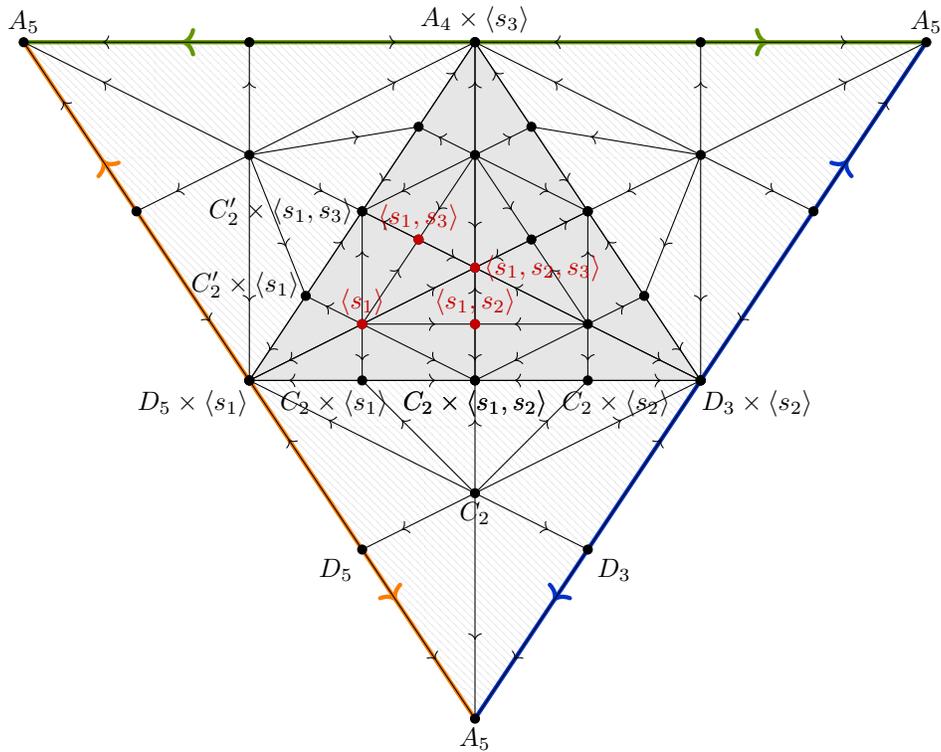
\begin{figure}[!h]
\centering
\begin{tikzpicture}[scale=0.75]

\definecolor{vlgray}{RGB}{230,230,230}


\definecolor{torange}{RGB}{250,125,000}
\definecolor{lorange}{RGB}{250,200,075}

\definecolor{vlgray}{RGB}{230,230,230}
\definecolor{tpurple}{RGB}{100,150,00}
\definecolor{lpurple}{RGB}{150,200,050}

\definecolor{blue}{RGB}{000,50,200}

\draw[pattern=north west lines, pattern color=gray, opacity=0.4] (-8,6)--(8,6)--(0,-6)--(-8,6);

\draw[fill, vlgray] (-4,0)--(4,0)--(0,6)--(-4,0);

\draw[ultra thick,tpurple, ->-] (0,6)--(-8,6);
\draw[ultra thick,tpurple, ->-] (0,6)--(8,6);
\draw[ultra thick,blue, ->-] (4,0)--(8,6);
\draw[ultra thick,blue, ->-] (4,0)--(0,-6);
\draw[ultra thick,torange, ->-] (-4,0)--(-8,6);
\draw[ultra thick,torange, ->-] (-4,0)--(0,-6);

\draw (-4,0)--(4,0);
\draw (-4,0)--(0,6);
\draw (4,0)--(0,6);

\draw (0,0)--(0,6);
\draw (-2,3)--(4,0);
\draw (-4,0)--(2,3);

\draw[ fill] (-4,0)  circle [radius=0.08];

\draw[ fill] (4,0)  circle [radius=0.08];

\draw[ fill] (0,6)  circle [radius=0.08];

\draw[ fill] (0,2)  circle [radius=0.08];)

\draw[ fill] (-2,3)  circle [radius=0.08];

\draw[ fill] (2,3)  circle [radius=0.08];

\draw[fill] (0,0)  circle [radius=0.08];



\draw[fill] (0,1)  circle [radius=0.08];
\draw[->-] (0,1)--(0,0);
\draw[->-] (0,1)--(0,2);

\draw[fill] (-1,2.5)  circle [radius=0.08];
\draw[->-] (-1,2.5)--(0,2);
\draw[->-] (-1,2.5)--(-2,3);

\draw[fill] (1,2.5)  circle [radius=0.08];
\draw[->-] (1,2.5)--(0,2);
\draw[->-] (1,2.5)--(2,3);

\draw[fill] (-2,1)  circle [radius=0.08];
\draw[->-] (-2,1)--(-4,0);
\draw[->-] (-2,1)--(-2,0);
\draw[->-] (-2,1)--(0,0);
\draw[->-] (-2,1)--(0,1);
\draw[->-] (-2,1)--(0,2);
\draw[->-] (-2,1)--(-1,2.5);
\draw[->-] (-2,1)--(-2,3);
\draw[->-] (-2,1)--(-3,1.5);

\draw[fill] (2,1)  circle [radius=0.08];

\draw[->-] (2,1)--(4,0);
\draw[->-] (2,1)--(2,0);
\draw[->-] (2,1)--(0,0);
\draw[->-] (2,1)--(0,1);
\draw[->-] (2,1)--(0,2);
\draw[->-] (2,1)--(1,2.5);
\draw[->-] (2,1)--(2,3);
\draw[->-] (2,1)--(3,1.5);

\draw[fill] (0,4)  circle [radius=0.08];

\draw[->-] (0,4)--(0,6);
\draw[->-] (0,4)--(1,2.5);
\draw[->-] (0,4)--(-2,3);
\draw[->-] (0,4)--(2,3);
\draw[->-] (0,4)--(0,2);
\draw[->-] (0,4)--(-1,2.5);
\draw[->-] (0,4)--(-1,4.5);
\draw[->-] (0,4)--(1,4.5);

\draw[fill] (-2,0)  circle [radius=0.08];
\draw[->-] (-2,0)--(-4,0);
\draw[->-] (-2,0)--(0,0);

\draw[fill] (2,0)  circle [radius=0.08];
\draw[->-] (2,0)--(4,0);
\draw[->-] (2,0)--(0,0);

\draw[fill] (-3,1.5)  circle [radius=0.08];
\draw[->-] (-3,1.5)--(-4,0);
\draw[->-] (-3,1.5)--(-2,3);

\draw[fill] (3,1.5)  circle [radius=0.08];
\draw[->-] (3,1.5)--(4,0);
\draw[->-] (3,1.5)--(2,3);

\draw[fill] (-1,4.5)  circle [radius=0.08];
\draw[->-] (-1,4.5)--(0,6);
\draw[->-] (-1,4.5)--(-2,3);

\draw[fill] (1,4.5)  circle [radius=0.08];
\draw[->-] (1,4.5)--(0,6);
\draw[->-] (1,4.5)--(2,3);


\begin{scope}[shift={(-4,6)},rotate={180},scale=1]

\draw[->-] (0,0)--(4,0);
\draw[->-] (0,0)--(-4,0);

\draw[->-] (2,3)--(4,0);
\draw[->-] (2,3)--(0,6);

\draw[->-] (0,2)--(0,0);
\draw[->-] (0,2)--(0,6);

\draw[->-] (0,2)--(-2,3);
\draw[->-] (0,2)--(-4,0);
\draw[->-] (0,2)--(4,0);
\draw[->-] (0,2)--(2,3);

\draw[->-] (0,2)--(-3,1.5);
\draw[->-] (0,2)--(-1,4.5);

\draw[ fill] (-4,0)  circle [radius=0.08];

\draw[ fill] (4,0)  circle [radius=0.08];

\draw[ fill] (0,6)  circle [radius=0.08];

\draw[ fill] (0,2)  circle [radius=0.08];)

\draw[ fill] (-2,3)  circle [radius=0.08];

\draw[ fill] (2,3)  circle [radius=0.08];

\draw[fill] (0,0)  circle [radius=0.08];

\end{scope}

\begin{scope}[shift={(4,6)},rotate={180},scale=1]

\draw[->-] (0,0)--(4,0);
\draw[->-] (0,0)--(-4,0);

\draw[->-] (-2,3)--(-4,0);
\draw[->-] (-2,3)--(0,6);

\draw[->-] (0,2)--(0,0);
\draw[->-] (0,2)--(0,6);

\draw[->-] (0,2)--(-2,3);
\draw[->-] (0,2)--(-4,0);
\draw[->-] (0,2)--(4,0);
\draw[->-] (0,2)--(2,3);

\draw[->-] (0,2)--(3,1.5);
\draw[->-] (0,2)--(1,4.5);

\draw[ fill] (-4,0)  circle [radius=0.08];

\draw[ fill] (4,0)  circle [radius=0.08];

\draw[ fill] (0,6)  circle [radius=0.08];

\draw[ fill] (0,2)  circle [radius=0.08];)

\draw[ fill] (-2,3)  circle [radius=0.08];

\draw[ fill] (2,3)  circle [radius=0.08];

\draw[fill] (0,0)  circle [radius=0.08];

\end{scope}

\begin{scope}[shift={(0,0)},rotate={180},scale=1]

\draw[->-] (-2,3)--(-4,0);
\draw[->-] (-2,3)--(0,6);

\draw[->-] (2,3)--(4,0);
\draw[->-] (2,3)--(0,6);

\draw[->-] (0,2)--(0,0);
\draw[->-] (0,2)--(0,6);

\draw[->-] (0,2)--(-2,3);
\draw[->-] (0,2)--(-4,0);
\draw[->-] (0,2)--(4,0);
\draw[->-] (0,2)--(2,3);

\draw[->-] (0,2)--(-2,0);
\draw[->-] (0,2)--(2,0);

\draw[ fill] (-4,0)  circle [radius=0.08];

\draw[ fill] (4,0)  circle [radius=0.08];

\draw[ fill] (0,6)  circle [radius=0.08];

\draw[ fill] (0,2)  circle [radius=0.08];)

\draw[ fill] (-2,3)  circle [radius=0.08];

\draw[ fill] (2,3)  circle [radius=0.08];

\draw[fill] (0,0)  circle [radius=0.08];

\end{scope}

\node [above]  at (8,6)   {$A_5$};
\node [above]  at (-8,6)   {$A_5$};
\node [below]  at (0,-6)   {$A_5$};
\node [above ]  at (0,6)   {$A_4 \times \langle s_3 \rangle$};
\node [below ]  at (0,0)   {$C_2 \times \langle s_1,s_2 \rangle$};
\node [below ]  at (-2.5,0)   {$C_2 \times \langle s_1 \rangle$};
\node [below ]  at (2.5,0)   {$C_2 \times \langle s_2 \rangle$};
\node [below ]  at (0,-2)   {$C_2$};
\node [below left]  at (-2,-3)   {$D_5$};

\node [below right]  at (2,-3)   {$D_3$};

\definecolor{purp}{RGB}{200,000,000}

\node [above, purp ]  at (-2,1)   {$ \langle s_1 \rangle$};

\node [above, purp]  at (0,1)   {$ \langle s_1,s_2 \rangle$};
\node [right, purp]  at (0,2)   {$ \langle s_1,s_2,s_3 \rangle$};
\node [above, purp]  at (-1,2.5)   {$ \langle s_1,s_3 \rangle$};

\draw[fill, purp] (-2,1)  circle [radius=0.08];
\draw[fill, purp] (0,1)  circle [radius=0.08];
\draw[fill, purp] (0,2)  circle [radius=0.08];

\draw[fill,purp] (-1,2.5)  circle [radius=0.08];

\node [left]  at (-2.95,1.7)   {$ C_2' \times \langle s_1 \rangle$};
\node [left]  at (-2,3)   {$ C_2' \times \langle s_1,s_3 \rangle$};

\node [below ]  at (0,0)   {$C_2 \times \langle s_1,s_2 \rangle$};

\node [below ]  at (-5,0)   {$D_5 \times \langle s_1 \rangle$};
\node [below]  at (5,0)   {$D_3 \times \langle s_2 \rangle$};




\end{tikzpicture}
\caption{A planar representation of the poset $\mathcal Q_{>U_0}$. The poset $\mathcal Q_{>U_0}$ is obtained by identifying pairs of green, blue and orange segments respectively. Assignment of some local subgroups is presented.}
\label{fig:reflikeposet}
\end{figure}

The poset $Q$ has the smallest element $U_0$ whose local subgroup is trivial, and the panel $\pan{\D}{U_0}$ is the entire complex $\D$. The subposet $\mathcal Q_{>U_0}$ is presented in Figure~\ref{fig:reflikeposet}. One verifies that $G(\mathcal Q)$ satisfies the assumptions of Definition~\ref{def:scofg}.(\ref{it:scofg2}). This, together with the fact that $D(\D, G(\mathcal Q)) \cong_G \Sigma_{W_L}$ is a model for $\ue G$ implies that we can apply Proposition~\ref{prop:bredondim} to calculate $\ucd G$.\end{example}

\begin{proposition}We have
\[2=\vcd G <\ucd G = \ugd G = \mathrm{dim}(\B) =\mathrm{dim}(\D)=3.\]
\end{proposition}

\begin{proof}First we show that $\ucd G=3$. From the description of the poset $\mathcal Q_{>U_0}$ (see Figure~\ref{fig:reflikeposet}) it follows that \[\D_{>{U_0}}=\abs{\mathcal{Q}_{>U_0}}=\partial C(Y') \cong S^2,\] and since $\rH^2({S^2}) \cong \mathbb{Z}$, by Proposition~\ref{prop:bredondim} we have $\ucd G \geqslant 3$.  One verifies that for any other element $U \in \mathcal Q$, one has $\mathrm{dim}(\D_{>U}) \leqslant 2$ and thus $\rH^n(\D_{>U})=0$ for $n> 2$. This implies that $\ucd G = 3$, and consequently that $\ugd G =\mathrm{dim}(\B)=3$. 

It remains to show that $\vcd G=2$. For this, note that $G=W_L \rtimes F$ is a finite extension of $W_L$ so $\vcd G=\vcd W_L$. Since $L$ is $2$--dimensional, acyclic and it has a finite, non-trivial fundamental group (see \cite[Remark, p.~10]{BLN}), by Proposition~\ref{prop:BLN} we obtain $\vcd W_L=2$.
\end{proof}

\subsection{Locally $6$--large complexes of groups}
Another theory which provides tools for ensuring (strict) developability of complexes of groups is the so-called \emph{simplicial non-positive curvature}. This theory can be seen as a combinatorial counterpart of the theory of metric non-positive curvature. We refer the reader to \cite{JS} for a detailed treatment of the subject. The key concept is that of \emph{local $6$--largeness} (i.e.,\ simplicial non-positive curvature). This is a combinatorial condition imposed on a simplicial complex $X$ which endows $X$ with many non-positive curvature-like properties.

In particular, there is a notion of locally $6$--large simple complex of finite groups \cite[Definition~6.2]{JS}. Similarly as in the case of non-positively curved complexes of groups, such complex $G(\mathcal{Q})$ is always developable. Moreover, if $K$ is a locally $6$--large simplicial complex, then the standard development $D(\D, G(\mathcal{Q}))$ admits a structure of a $6$--\emph{systolic} simplicial complex (that is, a simply-connected and locally $6$--large) \cite[Theorem~6.1]{JS}. It follows by \cite[Theorem~E]{CO} that $D(\D, G(\mathcal{Q}))$ is a model for $\ue G$, where $G$ is the fundamental group of $G(Q)$. Thus Theorem~\ref{thm:main} applies to $G(\mathcal{Q})$.

On the other hand, the only examples of groups constructed this way are the simplices of groups \cite[\S19]{JS}. It is not hard to see that for these, the standard development is a model for $\ue G$ of optimal dimension (where $G$ is a fundamental group of such simplex of groups).

\appendix

\section{Bestvina complex over a ring $R$}\label{sec:appen}

Here we present a version of the Bestvina complex $\B^R$, where $R$ is a suitably nice ring. The complex $\B^R$ in general is not contractible, and thus the basic construction $D(\B^R, G(\mathcal Q))$ is not a model for $\ue G$. However, both $\B^R$ and $D(\B^R, G(\mathcal Q)$) are $R$--acyclic and therefore on the level of chain complexes, the basic construction $D(\B^R, G(\mathcal Q))$ may be seen as a `model' for $\ue G$. The main point here is that the dimension of  $D(\B^R, G(\mathcal Q))$ is equal to $\ucd_R G$, the Bredon cohomological dimension of $G$ over the ring $R$. Before making this statement precise we need to recall some terminology. We refer to \cite{luckbook} for a detailed account of Bredon cohomology.

Let $R$ be a commutative ring with unit, let $G$ be a discrete group and let $\mathcal{F}$ be the family of all finite subgroups of $G$. The \emph{orbit category} $\orb$ (over $R$) is the category defined by the objects which are the left coset spaces $G/H$ with $H \in \mathcal{F}$ and the morphisms which are $G$--equivariant maps between the objects. An \emph{$\orb$--module} is a contravariant functor $M \colon \orb \rightarrow R\mbox{--Mod}$. The \emph{category of $\orb$--modules}, denoted by $\orbmod$, is the category whose objects are $\orb$--modules and whose morphisms are natural transformations between these objects. A sequence \[0\rightarrow M' \rightarrow M \rightarrow M'' \rightarrow 0\]
in $\orbmod$ is called {\it exact} if it is exact after evaluating in $G/H$ for each $H \in \mathcal{F}$.  Let $M \in \orbmod$ and consider the left exact functor
\[ \nathom(M,-) \colon \orbmod \rightarrow R\mbox{--Mod} : N \mapsto \nathom(M,N), \]
where $\nathom(M,N)$ is the $R$--module of all natural transformations from $M$ to $N$. The module $M$ is said to be a \emph{projective $\orb$--module} if and only if this functor is exact. The module $F\in\orbmod$ is said to be a {\it free} $\orb$--module if $F\cong \bigoplus_{K\in \mathcal{I}} R[\mbox{map}_G(-, G/K)]$ for some subset $\mathcal{I}\subseteq \mathcal{F}$. It is not difficult to check that free modules are projective. It can also be shown that $\orbmod$ contains enough projective modules to construct projective resolutions. The resulting  $\mathrm{Ext}$--functors $\mathrm{Ext}^{n}_{\orb}(-,M)$ will have all the usual properties. The \emph{$n$--th Bredon cohomology of $G$ over $R$} with coefficients $M \in \orbmod$ is by definition
\[ \mathrm{H}^n_{\mathcal{F}}(G,M)= \mathrm{Ext}^{n}_{\orb}(\underline{R},M), \]
where $\underline{R}$ is the functor that maps all objects to $R$ and all morphisms to the identity map.  The \emph{Bredon cohomological dimension of $G$ over $R$} is defined to be 
\[ \ucd_R G = \sup\{ n \in \mathbb{N} \ | \ \exists M \in \orbmod :  \mathrm{H}^n_{\mathcal{F}}(G,M)\neq 0 \}. \]
Given a $G$--CW--complex $X$, a Bredon module \[\underline{C}_{n}(X, R)(-) \colon \mathcal O_{\mathcal{F}}G \to R\mbox{--Mod}\] is defined as \[\underline{C}_{n}(X, R)(G/H) = C_{n}(X^H, R),\] where $C_{\ast}(-,R)$ denotes the cellular chains with coefficients in $R$. Note that, in this way, the augmented cellular chain complex over $R$ of any model for $\ue G$ yields a free resolution of $\underline{R}$ which can then  be used to compute $\mathrm{H}_{\mF}^{\ast}(G,-)$. It follows that $\ucd_R G\leqslant \ugd G$.

\begin{theorem}\label{thm:algthm}
Let $G(\mathcal Q)$ be a strictly developable simple complex of finite groups with the fundamental group $G$, and let $R$ be either a prime field or a subring of $\mathbb Q$ that contains $1$. Then there is a panel complex $(\B^R, \{\B^R_J\}_{J \in \mathcal Q})$ over $\mathcal{Q}$ such that

\begin{enumerate} 
\item \label{it:algthm1} there is a chain homotopy equivalence of Bredon chain complexes
\[\underline{C}_{\ast}(D(\B^R, G(\mathcal{Q})), R) \to \underline{C}_{\ast}(D(\D, G(\mathcal{Q})),R).\]

\item \label{it:algthm2} if $D(\D, G(\mathcal{Q}))$ is  a model for $\underline{E}G$ then $\mathrm{dim}(D(\Br, G(\mathcal Q)))=
\ucd_R G$ and the Bredon chain complex  $\underline{C}_{\ast}(D(\B^R, G(\mathcal Q)), R)$ gives a free resolution of $\underline{R}$ of length equal to $\ucd_R G$. 
\end{enumerate} 
\end{theorem}

Both, the construction of the complex $\B^R$, and the proofs of parts 
(\ref{it:algthm1}) and (\ref{it:algthm2}) of the above theorem are analogous to the constructions and proofs performed in Sections~\ref{sec:scofgs} and \ref{sec:sdabc}. The general principle is that at any step one replaces `contractible' with `$R$--acyclic' and instead of considering homotopy equivalences of $\mathrm{CW}$--complexes, one considers chain homotopy equivalences of chain complexes. 

\begin{proof}[Outline of the proof of Theorem~\ref{thm:algthm}]
 \noindent
 \vskip 0.1cm

\noindent
\textbf{Step 1.} Definition of the complex $\B^R$.\smallskip

To define $\Br$ we proceed the same as in Definition~\ref{def:bestcx}, except that instead of taking a contractible polyhedron, we take a compact $R$--acyclic polyhedron of the smallest dimension that contains $\rlink{\B^R}{J}$. By \cite[Lemma~24]{LeSa} the condition for existence of such polyhedron of dimension equal to $\mathrm{dim}(\rlink{\B^R}{J})=n$ is that \[\rH^n(\rlink{\B^R}{J},R)=0.\] Note that as opposed to Proposition~\ref{prop:dimbest}, here we allow $n=2$. 
Consequently, proceeding as in Proposition~\ref{prop:dimbest} we obtain 
\begin{equation}\label{eq:dimbestalg}\mathrm{dim}(B^R) =\mathrm{max}\{n \in \mathbb{N} \mid \rH^{n-1}\big(\cup_{J < J'}\rpan{\B^R}{J'}, R\big) \neq 0 \text{ for some } J \in \mathcal Q \}.\end{equation}
We remark that the assumption that $R$ is a prime field or a subring of $\mathbb Q$ that contains $1$ is needed in \cite[Lemma~24]{LeSa}.\medskip 

\noindent
\textbf{Step 2.} Existence and homotopy uniqueness of panel maps $ f \colon \C{\ast}{\Br} \to \C{\ast}{\D}$ and $ g \colon \C{\ast}{\D} \to \C{\ast}{\Br}$.\smallskip

For chain complexes $\C{\ast}{X}$ and $\C{\ast}{Y}$  over panel complexes $(X, \{X_J\}_{J \in \mathcal Q})$ and  $(Y, \{Y_J\}_{J \in \mathcal Q})$ by a \emph{panel chain map} we mean a chain map $f \colon \C{\ast}{X} \to \C{\ast}{Y}$ that for any panel $\rpan{X}{J}$ restricts to a chain map  $f_J \colon \C{\ast}{\rpan{X}{J}} \to \C{\ast}{\rpan{Y}{J}}$.

Recall that a \emph{chain homotopy} between chain maps $f,g \colon \C{\ast}{X} \to \C{\ast}{Y}$ is a sequence of maps $\psi_n \colon \C{n}{X} \to \C{n+1}{Y}$ such that for any $n \geqslant 0$ we have $f_n-g_n=d  \psi_n + \psi_{n-1} d$, where $f_n$ (resp.\ $g_n$) denote the restrictions of $f$ (resp.\ $g$) to $\C{n}{X}$. We say that $\{\psi_n\}_{n\geqslant 0}$ is a \emph{panel chain homotopy} if for every panel $\pan{X}{J}$ the map $\psi_n$ restricts to the map $\psi_n|_J \colon \C{n}{\pan{X}{J}} \to \C{n+1}{\pan{Y}{J}}$.\medskip

We have the following chain complex-analogue of Lemma~\ref{lem:panelhtpy}.

\begin{lemma}\label{lem:algpanelhtpy}
Let $(Y, \{\pan{Y}{J}\}_{J\in \mathcal Q})$ be a panel complex over $\mathcal Q$ such that for every $J \in \mathcal Q$ the panel $\pan{Y}{J}$ is $R$--acyclic. Then for any panel complex $(X, \{\pan{X}{J}\}_{J\in \mathcal Q})$ there is a panel chain map $\C{\ast}{X} \to \C{\ast}{Y}$ which is unique up to panel chain homotopy.
\end{lemma}

\begin{proof}The proof is essentially the same as that of Lemma~\ref{lem:panelhtpy}. We define the map on the basis of $\C{\ast}{X}$ (i.e.,\ on cells of $X$) and then extend it $R$--linearly to $\C{\ast}{X}$. In order to construct the map one proceeds by induction over panels, and for a given panel $\pan{X}{J}$, by induction on the dimension of cells in $\pan{X}{J}$. At any step, in order to extend the map to a given cell, one uses the fact that the target panel $\pan{Y}{J}$ is $R$--acyclic.   
\end{proof}

Lemma~\ref{lem:algpanelhtpy} implies the existence of claimed maps $ f \colon \C{\ast}{\Br} \to \C{\ast}{\D}$ and $ g \colon \C{\ast}{\D} \to \C{\ast}{\Br}$. Moreover, it implies that the composites $f \circ g$ and $g \circ f$ are panel chain homotopic to the identity maps on $\D$ and $\Br$ respectively. Let $\{\psi_n\}_{n\geqslant 0}$  and $\{\phi_n\}_{n\geqslant 0}$ denote the respective chain homotopies.\medskip

\noindent
\textbf{Step 3.} Proof of (\ref{it:algthm1}). \smallskip

First we need the following analogue of Lemma~\ref{lem:ghtpyeq}.

\begin{lemma}\label{lem:algghtpyeq}
Let $G(\mathcal Q)$ be a simple complex of finite groups over the poset $\mathcal{Q}$ and let $(X, \{\pan{X}{J}\}_{J\in \mathcal Q})$ and $(Y, \{\pan{Y}{J}\}_{J\in \mathcal Q })$ be two panel complexes over $\mathcal Q$. If $\C{\ast}{X}$ and $\C{\ast}{Y}$ are panel chain homotopy equivalent then $\C{\ast}{D(X, G(\mathcal{Q}))}$ and $\C{\ast}{D(Y, G(\mathcal{Q}))}$ are $G$--chain homotopy equivalent.\end{lemma}
Here the $G$--action on $\C{\ast}{D(X, G(\mathcal{Q}))}$ (resp.\ $\C{\ast}{D(Y, G(\mathcal{Q}))}$) is induced by the $G$--action on $D(X, G(\mathcal{Q}))$ (resp.\ $D(Y, G(\mathcal{Q}))$.

\begin{proof}The proof is the same as that of Lemma~\ref{lem:ghtpyeq}.\end{proof}

Given maps $f$ and $g$ and homotopies $\{{\psi}_n\}_{n\geqslant 0}$ and $\{\phi_n\}_{n\geqslant 0}$ constructed in Step~2, Lemma~\ref{lem:algghtpyeq} gives $G$--chain maps
\begin{align*} 
\tilde{f} & \colon C_{\ast}(D(\Br, G(\mathcal{Q})),R) \to C_{\ast}(D(\D, G(\mathcal{Q})),R),\\
\tilde{g} & \colon C_{\ast}(D(\D, G(\mathcal{Q})),R) \to C_{\ast}(D(\Br, G(\mathcal{Q})),R), \end{align*}
and $G$--chain homotopies
\begin{align*}
\{\tilde{\phi}_n\}_{n\geqslant 0} & \colon C_{\ast}(D(\Br, G(\mathcal{Q})),R) \to C_{\ast}(D(\Br, G(\mathcal{Q})),R),\\
\{\tilde{\psi}_n\}_{n\geqslant 0} & \colon C_{\ast}(D(\D, G(\mathcal{Q})),R) \to C_{\ast}(D(\D, G(\mathcal{Q})),R). 
\end{align*}
Since the $G$--action on $\C{\ast}{D(\Br, G(\mathcal{Q}))}$ (resp.\ $\C{\ast}{D(\D, G(\mathcal{Q}))}$) is induced by the one on $D(\Br, G(\mathcal{Q}))$ (resp.\ $D(\D, G(\mathcal{Q}))$, for any finite subgroup $H \subset G$ all the above maps restrict to the subcomplex $\C{\ast}{D(\Br, G(\mathcal{Q}))^H}$ (resp. $\C{\ast}{D(\D, G(\mathcal{Q}))^H}$), thus ensuring that the restriction
\[\tilde{f}^H \colon C_{\ast}(D(\Br, G(\mathcal{Q}))^H,R) \to C_{\ast}(D(\D, G(\mathcal{Q}))^H,R)\]
is a chain homotopy equivalence. One verifies that the restrictions $\tilde{f}^H$ for various $H \subset G$ are compatible with the morphisms in the category $\orb$ and thus \[\tilde{f}^{(-)} \colon \underline{C}_{\ast}(D(\B^R, G(\mathcal{Q})), R)(-) \to \underline{C}_{\ast}(D(\D, G(\mathcal{Q})),R)(-)\] is the required chain homotopy equivalence of Bredon chain complexes.
\medskip

\noindent
\textbf{Step 4.} Proof of (\ref{it:algthm2}).\smallskip

Since $D(\D, G(\mathcal{Q}))$ is a model for $\ue G$ we get that for any finite subgroup $H \subset G$ the fixed point set $D(\D, G(\mathcal{Q}))^H$ is contractible, and thus in particular the chain complex $C_{\ast}(D(\D, G(\mathcal{Q}))^H, R)
$ is $R$--acyclic. By (\ref{it:algthm1}) we get a chain homotopy equivalence \[C_{\ast}(D(\Br, G(\mathcal{Q}))^H, R) \to C_{\ast}(D(\D, G(\mathcal{Q}))^H,R)\] and therefore the chain complex $C_{\ast}(D(\Br, G(\mathcal{Q}))^H, R)$ is $R$--acyclic as well. Thus we obtain that $\underline{C}_{\ast}(D(\Br, G(\mathcal Q)), R)$ is a free resolution of $\underline{R}$ of length equal to $\mathrm{dim}(D(\Br, G(\mathcal Q)))=\mathrm{dim}(\Br) $.

It remains to show that $\mathrm{dim}(\B^R)= \ucd_R G$. By \eqref{eq:dimbestalg} we have \begin{equation*}\mathrm{dim}(B^R) =\mathrm{max}\{n \in \mathbb{N} \mid \rH^{n-1}\big(\cup_{J < J'}\rpan{\B^R}{J'}, R\big) \neq 0 \text{ for some } J \in \mathcal Q \}.\end{equation*} 
Using panel maps $f$ and $g$ constructed in Step~2 of this proof, by restricting them to $
\C{\ast}{\rlink{\Br}{J}}$ and $\C{\ast}{\flink{\D}{J}}$ respectively, we get that $
\C{\ast}{\rlink{\Br}{J}}$ and $\C{\ast}{\flink{\D}{J}}$ are chain homotopy equivalent. Thus we obtain \begin{equation}\label{eq:dimbestalg2}\mathrm{dim}(B^R) =\mathrm{max}\{n \in \mathbb{N} \mid \rH^{n-1}\big(\cup_{J < J'}\rpan{\D}{J'}, R\big) \neq 0 \text{ for some } J \in \mathcal Q \}.\end{equation}

Note that if $R=\mathbb{Z}$ then the right-hand side of the above formula is equal to (\ref{eq:dimdavis}) and thus equal to $\ucd G$. The claim is that the same holds for the ring $R$. For this we need to show that the formula \eqref{eq:dimdieter} holds over $R$, i.e.,\ that we have \begin{equation}\label{eq:dimgeneraliseddieter}\ucd_R G = \mathrm{max}\{ n \in \mathbb{N} \mid \rH^n(K_{\Omega_J},K_{>\Omega_J}, R) \neq 0 \text{ for some } J \in \mathcal Q\}.\end{equation} One verifies that the proof of the above formula in \cite[Theorem 5.1]{DMP} carries through over $R$. Given \eqref{eq:dimgeneraliseddieter}, by applying Proposition~\ref{prop:bredondim} (over $R$) we obtain that the right-hand side of \eqref{eq:dimbestalg2} is equal to $\ucd_R G$.
\end{proof}




\begin{remark}Note that, as opposed to $\B$, for $\Br$ the equality $\mathrm{dim}(\Br)=\ucd_R G$ holds also when $\ucd_R G=2$. In particular, the complex $\B$ may be different from $ B^{\mathbb{Z}}$.
\end{remark}

Observe that a $1$--dimensional $\mathrm{CW}$--complex is $R$--acyclic if and only if it is contractible. Thus if $\ucd_R G \leqslant 1$ we get that $\Br = \B$ and hence $ \ucd G \leqslant 1$. Therefore we obtain the following strengthened version of Theorem~\ref{thm:virtfree}.

\begin{corollary}Let $G(\mathcal Q)$ be a strictly developable simple complex of finite groups over the poset $\mathcal{Q}$ with the fundamental group $G$. Suppose that $D(\D, G(\mathcal{Q}))$ is  a model for $\underline{E}G$ and that $\ucd_R  G\leqslant 1$. Then $D(\D, G(\mathcal{Q}))$ equivariantly deformation retracts onto the tree $D(\B, G(\mathcal{Q}))$.
\end{corollary}


\begin{bibdiv}

\begin{biblist}



\bib{Ash}{article}{
    AUTHOR = {Ash, Avner},
     TITLE = {Deformation retracts with lowest possible dimension of
              arithmetic quotients of self-adjoint homogeneous cones},
   JOURNAL = {Math. Ann.},
  FJOURNAL = {Mathematische Annalen},
    VOLUME = {225},
      YEAR = {1977},
    NUMBER = {1},
     PAGES = {69--76},
}

\bib{BCH}{article}{
    AUTHOR = {Baum, Paul},
    AUTHOR = {Connes, Alain},
    AUTHOR = {Higson, Nigel},
     TITLE = {Classifying space for proper actions and {$K$}-theory of group
              {$C^\ast$}-algebras},
 BOOKTITLE = {{$C^\ast$}-algebras: 1943--1993 ({S}an {A}ntonio, {TX}, 1993)},
    SERIES = {Contemp. Math.},
    VOLUME = {167},
     PAGES = {240--291},
 PUBLISHER = {Amer. Math. Soc., Providence, RI},
      YEAR = {1994},
}


\bib{Best}{article}{
   author={Bestvina, Mladen},
   title={The virtual cohomological dimension of Coxeter groups},
   conference={
      title={Geometric group theory, Vol.\ 1},
      address={Sussex},
      date={1991},
   },
   book={
      series={London Math. Soc. Lecture Note Ser.},
      volume={181},
      publisher={Cambridge Univ. Press, Cambridge},
   },
   date={1993},
   pages={19--23},
}

\bib{BLN}{article}{
   author={Brady, Noel},
   author={Leary, Ian J.},
   author={Nucinkis, Brita E. A.},
   title={On algebraic and geometric dimensions for groups with torsion},
   journal={J. London Math. Soc. (2)},
   volume={64},
   date={2001},
   number={2},
   pages={489--500},
}
    


\bib{BH}{book}{
  author={Bridson, Martin R.},
  author={Haefliger, Andr{\'e}},
  title={Metric spaces of non-positive curvature},
  series={Grundlehren der Mathematischen Wissenschaften [Fundamental
    Principles of Mathematical Sciences]},
  volume={319},
  publisher={Springer-Verlag, Berlin},
  date={1999},
  pages={xxii+643},
}

\bib{brownwall}{article}{
    AUTHOR = {Brown, Kenneth S.},
     TITLE = {Groups of virtually finite dimension},
 BOOKTITLE = {Homological group theory ({P}roc. {S}ympos., {D}urham, 1977)},
    SERIES = {London Math. Soc. Lecture Note Ser.},
    VOLUME = {36},
     PAGES = {27--70},
 PUBLISHER = {Cambridge Univ. Press, Cambridge-New York},
      YEAR = {1979},
}  

\bib{brownco}{book}{
    AUTHOR = {Brown, Kenneth S.},
     TITLE = {Cohomology of groups},
    SERIES = {Graduate Texts in Mathematics},
    VOLUME = {87},
 PUBLISHER = {Springer-Verlag, New York-Berlin},
      YEAR = {1982},
     PAGES = {x+306},
      ISBN = {0-387-90688-6},
   MRCLASS = {20-02 (18-01 20F32 20J05 55-01)},
  MRNUMBER = {672956},
MRREVIEWER = {Ross Staffeldt},
}


\bib{CO}{article}{
   author={Chepoi, Victor},
   author={Osajda, Damian},
   title={Dismantlability of weakly systolic complexes and applications},
   journal={Trans. Amer. Math. Soc.},
   volume={367},
   date={2015},
   number={2},
   pages={1247--1272},
}


\bib{Davbuild}{article}{
   author={Davis, Michael W.},
   title={Buildings are $\mathrm{CAT}(0)$},
   conference={
      title={Geometry and cohomology in group theory},
      address={Durham},
      date={1994},
   },
   book={
      series={London Math. Soc. Lecture Note Ser.},
      volume={252},
      publisher={Cambridge Univ. Press, Cambridge},
   },
   date={1998},
   pages={108--123},
}


\bib{Davbook}{book}{
   author={Davis, Michael W.},
   title={The geometry and topology of Coxeter groups},
   series={London Mathematical Society Monographs Series},
   volume={32},
   publisher={Princeton University Press, Princeton, NJ},
   date={2008},
   pages={xvi+584},
}

\bib{Davis}{article}{
    AUTHOR = {Davis, Michael W.},
     TITLE = {Groups generated by reflections and aspherical manifolds not
              covered by {E}uclidean space},
   JOURNAL = {Ann. of Math. (2)},
  FJOURNAL = {Annals of Mathematics. Second Series},
    VOLUME = {117},
      YEAR = {1983},
    NUMBER = {2},
     PAGES = {293--324},
      ISSN = {0003-486X},
}

\bib{DMP}{article}{
   author={Degrijse, Dieter},
   author={Mart\'\i nez-P\'erez, Conchita},
   title={Dimension invariants for groups admitting a cocompact model for
   proper actions},
   journal={J. Reine Angew. Math.},
   volume={721},
   date={2016},
   pages={233--249},
   issn={0075-4102},
}


\bib{Dun}{article}{
    AUTHOR = {Dunwoody, M. J.},
     TITLE = {Accessibility and groups of cohomological dimension one},
   JOURNAL = {Proc. London Math. Soc. (3)},
  FJOURNAL = {Proceedings of the London Mathematical Society. Third Series},
    VOLUME = {38},
      YEAR = {1979},
    NUMBER = {2},
     PAGES = {193--215},
}


 
 
\bib{Har}{article}{
    AUTHOR = {Harer, John L.},
     TITLE = {The virtual cohomological dimension of the mapping class group
              of an orientable surface},
   JOURNAL = {Invent. Math.},
  FJOURNAL = {Inventiones Mathematicae},
    VOLUME = {84},
      YEAR = {1986},
    NUMBER = {1},
     PAGES = {157--176},
}

\bib{HarMei}{article}{
   author={Harlander, Jens},
   author={Meinert, Holger},
   title={Higher generation subgroup sets and the virtual cohomological
   dimension of graph products of finite groups},
   journal={J. London Math. Soc. (2)},
   volume={53},
   date={1996},
   number={1},
   pages={99--117},
}


\bib{Hat}{book}{
   author={Hatcher, Allen},
   title={Algebraic topology},
   publisher={Cambridge University Press, Cambridge},
   date={2002},
   pages={xii+544},
}

\bib{JS}{article}{
   author={Januszkiewicz, Tadeusz},
   author={{\'S}wi{\c a}tkowski, Jacek},
   title={Simplicial nonpositive curvature},
   journal={Publ. Math. Inst. Hautes \'Etudes Sci.},
   number={104},
   date={2006},
   pages={1--85},
}

\bib{LePe}{article}{
    AUTHOR = {Leary, Ian J.},
    AUTHOR={Petrosyan, Nansen},
     TITLE = {On dimensions of groups with cocompact classifying spaces for
              proper actions},
   JOURNAL = {Adv. Math.},
  FJOURNAL = {Advances in Mathematics},
    VOLUME = {311},
      YEAR = {2017},
     PAGES = {730--747},
}

\bib{LeSa}{article}{
   author={Leary, Ian J.},
   author={Saadeto\u glu, M\"uge},
   title={The cohomology of Bestvina-Brady groups},
   journal={Groups Geom. Dyn.},
   volume={5},
   date={2011},
   number={1},
   pages={121--138},
   issn={1661-7207},
   review={\MR{2763781}},
}

\bib{luckbook}{book}{
    AUTHOR = {L\"uck, W.},
     TITLE = {Transformation groups and algebraic {$K$}-theory},
    SERIES = {Lecture Notes in Mathematics},
    VOLUME = {1408},
      NOTE = {Mathematica Gottingensis},
 PUBLISHER = {Springer-Verlag, Berlin},
      YEAR = {1989},
}


\bib{Lucksurvey}{article}{,
    AUTHOR = {L\"uck, W.},
     TITLE = {Survey on classifying spaces for families of subgroups},
 BOOKTITLE = {Infinite groups: geometric, combinatorial and dynamical
              aspects},
    SERIES = {Progr. Math.},
    VOLUME = {248},
     PAGES = {269--322},
 PUBLISHER = {Birkh\"auser, Basel},
      YEAR = {2005},
}

\bib{LuckMeintrup}{article}{,
    AUTHOR = {L\"uck, W.}
    author= {Meintrup, D.},
     TITLE = {On the universal space for group actions with compact
              isotropy},
 BOOKTITLE = {Geometry and topology: {A}arhus (1998)},
    SERIES = {Contemp. Math.},
    VOLUME = {258},
     PAGES = {293--305},
 PUBLISHER = {Amer. Math. Soc., Providence, RI},
      YEAR = {2000},
}

\bib{Vog}{article}{,
    AUTHOR = {Vogtmann, Karen},
     TITLE = {Automorphisms of free groups and outer space},
 BOOKTITLE = {Proceedings of the {C}onference on {G}eometric and
              {C}ombinatorial {G}roup {T}heory, {P}art {I} ({H}aifa, 2000)},
   JOURNAL = {Geom. Dedicata},
  FJOURNAL = {Geometriae Dedicata},
    VOLUME = {94},
      YEAR = {2002},
     PAGES = {1--31},
}

\end{biblist}
\end{bibdiv}
\end{document}